\documentclass[leqno,11pt,a4paper]{amsart} 
\makeatletter
\renewcommand\section{\@startsection{section}{1}{\z@}%
           {25\p@ \@plus 6\p@ \@minus 3\p@}%
           {10\p@ \@plus 6\p@ \@minus 3\p@}%
           {\fontsize{13pt}{0cm}\selectfont\bfseries\boldmath}}
\renewcommand\subsection{\@startsection{subsection}{2}{\z@}%
           {13\p@ \@plus 6\p@ \@minus 3\p@}%
           {6\p@ \@plus 6\p@ \@minus 3\p@}%
           {\fontsize{12pt}{0cm}\itshape}}
\renewcommand\subsubsection{\@startsection{subsubsection}{3}{\z@}%
           {12\p@ \@plus 6\p@ \@minus 3\p@}%
           {\p@}%
           {\normalfont\normalsize\itshape}}
\renewcommand{\paragraph}[1]{%
  \par
  \addvspace{\medskipamount}
  \noindent
  \textbf{#1\@addpunct{.}}\enspace\ignorespaces
}
\makeatother
\let\oldtocsection=\tocsection
\let\oldtocsubsection=\tocsubsection
\let\oldtocsubsubsection=\tocsubsubsection
\renewcommand{\tocsection}[2]{
\hspace{0em}\bfseries\oldtocsection{#1}{#2}}
\renewcommand{\tocsubsection}[2]{\hspace{1em}\small\oldtocsubsection{#1}{#2}}
\renewcommand{\tocsubsubsection}[2]{\hspace{2em}\small\oldtocsubsubsection{#1}{#2}}

\usepackage[colorlinks=true,urlcolor=blue,citecolor=red,linkcolor=blue,linktocpage,pdfpagelabels,bookmarksnumbered,bookmarksopen]{hyperref}
\usepackage{amssymb, amsmath, amsthm, physics}
\usepackage{amsfonts} 
\usepackage{color}
\usepackage{mathrsfs}
\usepackage{graphicx}
\usepackage{comment}
\usepackage{latexsym, mathtools}
\usepackage{xparse}
\usepackage{layout}
\usepackage{fancyhdr}
\usepackage[capitalize,nameinlink]{cleveref}
\crefname{subsection}{Subsection}{Subsection}
\usepackage{enumerate}
\DeclareSymbolFont{EulerExtension}{U}{euex}{m}{n}
\DeclareMathSymbol{\euintop}{\mathop} {EulerExtension}{"52}
\DeclareMathSymbol{\euointop}{\mathop} {EulerExtension}{"48}

\DeclareSymbolFont{cmletters}{OML}{cmm}{m}{it}
\DeclareSymbolFont{cmsymbols}{OMS}{cmsy}{m}{n}
\DeclareSymbolFont{cmlargesymbols}{OMX}{cmex}{m}{n}
\DeclareMathSymbol{\myjmath}{\mathord}{cmletters}{"7C}
\DeclareMathSymbol{\myamalg}{\mathbin}{cmsymbols}{"71}
\DeclareMathSymbol{\mycoprod}{\mathop}{cmlargesymbols}{"60}
\let\jmath\myjmath

\newtheorem{theorem}{Theorem}[section]
\newtheorem{proposition}[theorem]{Proposition}
\newtheorem{lemma}[theorem]{Lemma}
\newtheorem{corollary}[theorem]{Corollary}

\theoremstyle{definition}
\newtheorem{remark}[theorem]{Remark}
\newtheorem{definition}[theorem]{Definition}

\newtheorem*{claim}{Claim}
\newtheorem{step}{Step}
\newtheorem*{settings}{Settings}
\expandafter\let\expandafter\oldproof\csname\string\proof\endcsname
\let\oldendproof\endproof
\renewenvironment{proof}[1][\proofname]{%
  \oldproof[\bfseries\itshape #1] 
}{\oldendproof}
\numberwithin{equation}{section}
\renewcommand{\theequation}{\thesection.\arabic{equation}}

\allowdisplaybreaks[4]

\newcommand{\toinfty}{\to\infty}
\newcommand{\dsp}{\displaystyle}   
\newcommand{\eps}{\varepsilon}          
\renewcommand{\l}{\left}
\renewcommand{\r}{\right}
\newcommand{\la}{\left\langle}
\newcommand{\ra}{\right\rangle}
\newcommand{\R}{{\bf R}}
\newcommand{\N}{{\bf N}}
\newcommand{\NN}{{\bf Z}_{\ge 0}}
\newcommand{\Z}{{\bf Z}}
\newcommand{\RN}{\R^N}
\newcommand{\ZN}{{\bf Z}^N}
\newcommand{\0}{\setminus\{0\}}

\renewcommand{\d}{\partial}

\newcommand{\intrn}{\int_{\RN}}

\newcommand{\s}{\,\,\,}
\newcommand{\eqntag}{\addtocounter{equation}{1}\tag{\theequation}}

\newcommand{\Hrn}{H^1(\RN)}

\newcommand{\Hdrn}{\dot{H}^1(\RN)}

\newcommand{\Wdrn}{\dot{W}^{1,p}(\RN)}
\newcommand{\Wmp}{W^{m,p}(\RN)}

\newcommand{\Hdsp}{\dot{H}^{s,p}(\RN)}
\newcommand{\2}{2^*}   
\newcommand{\p}{p^*}
\newcommand{\ppm}{p^*_m}        
\newcommand{\dsl}[3]{#1_{#2,#3}}        
\newcommand{\bbl}[2]{#1^{#2}}                   
\newcommand{\Anrl}{ \mathcal{A}_{n,R,L} }

\newcommand{\inv}[1]{ #1^{-1} }

\begin{document}


\title[Profile Decomposition in inhomogeneous Sobolev spaces]{Profile decomposition in Sobolev spaces and decomposition of integral functionals~I: \\ inhomogeneous case} 

\author[M.~Okumura]{Mizuho Okumura} 




\renewcommand{\thefootnote}{\fnsymbol{footnote}}
\footnote[0]{2020\textit{ Mathematics Subject Classification}.
46B50, 46B06, 46E35.}

\address{
Graduate School of Science \endgraf
Tohoku University \endgraf
Sendai 980-8578 \endgraf
Japan
}
\email{okumura.mizuho.p3@dc.tohoku.ac.jp}




\maketitle

\begin{abstract}
The present paper is devoted to analyzing the lack of compactness of bounded sequences in \emph{inhomogeneous} Sobolev spaces, where bounded sequences might fail to be compact due to an isometric group action, that is, \emph{translation}. It will be proved that every bounded sequence $(u_n)$ has (possibly infinitely many) \emph{profiles}, and then the sequence is asymptotically decomposed into a sum of translated profiles and a double-suffix residual term, where the residual term becomes arbitrarily small in appropriate Lebesgue or Sobolev spaces of lower order. To this end, functional analytic frameworks are established in an abstract way by making use of a group action $G$, in order to characterize profiles by $(u_n)$ with $G$. One also finds that a decomposition of the Sobolev norm into profiles is bounded by the supremum of the norm of $u_n$. Moreover, the profile decomposition leads to results of decomposition of integral functionals of subcritical order. It is noteworthy that the space where the decomposition of integral functionals holds is the same as that where the residual term is vanishing.
\end{abstract}

{\small %
\tableofcontents
}


\section{Introduction}\label{section;intro}

\subsection{Reviews of known results}\label{subsection;prologue}
\paragraph{Prologue}
Profile decomposition is a method of asymptotic analysis especially of sequences in function spaces such as Lebesgue spaces, Sobolev spaces, Besov spaces, Triebel-Lizorkin spaces and so on, where bounded sequences may have defect of compactness, i.e., bounded sequences do not necessarily have strongly convergent subsequences. 
In this paper, we revisit profile decomposition in specific Banach spaces, i.e., Hilbert spaces and inhomogeneous Sobolev spaces.

Compactness plays a crucial role,  e.g., in studying convergence of approximate solutions to PDEs and convergence of minimizing sequences for energy functionals in variational problems. Hence the defect of compactness of functional sequences is a critical problem of great importance. 
Profile decomposition is a generic method of analyzing such sequences, revealing what prevents sequences from converging strongly.

Before recalling what profile decomposition is, let us briefly review classical ideas of analysis of non-compact sequences.
H.~Brezis and E.~Lieb~\cite{B-L}
showed that, for a bounded sequence $(u_n)$ in $L^p(\RN) \ (1\le p< \infty)$
with its pointwise limit $u_0\in L^p(\RN)$,  $\|u_n\|_{L^p(\RN)}^p$  is asymptotically decomposed into two pieces of $\|u_0\|_{L^p(\RN)}^p$ and $\|u_n - u_0\|_{L^p(\RN)}^p$, 
where the defect of compactness is described as $u_n-u_0$. This is called the Brezis-Lieb lemma and is interpreted as a refinement of the Fatou lemma, since it picks up a missing term $u_n-u_0$ to obtain the equality  $\|u_n\|_{L^p(\RN)}^p = \|u_0\|_{L^p(\RN)}^p + \|u_n - u_0\|_{L^p(\RN)}^p + o(1)$, instead of the inequality in the Fatou lemma. 
P.L.~Lions in~\cite{Lions 84 1,Lions 84 2, Lions 85 1, Lions 85 2}
studied the lack of compactness 
in Sobolev spaces caused by scaling  invariance properties of translations and dilations. 
His famous and celebrated ideas in~\cite{Lions 84 1, Lions 84 2}, named the concentration-compactness principle, classifies every sequence of probability measures into three disjoint situations: the sequence is vanishing (\emph{vanishing}); the sequence is tight (\emph{compactness}); masses of measures are asymptotically splitting into at least two pieces (\emph{dichotomy}). The second concentration-compactness principle in~\cite{Lions 85 1, Lions 85 2} provides more precise information provided that the sequence is in $\Hdrn$. 
Lions' theory has many applications to a variety of studies of mathematical analysis, e.g., PDEs, Functional Analysis and Calculus of Variations.  
One fundamental and significant applications of those works to variational problems is a minimization problem subordinated to Sobolev embeddings 
\begin{gather*}
\Hrn \hookrightarrow L^p(\RN) \quad (p\in ]2,\2[), \\
\Hdrn \hookrightarrow L^{\2}(\RN) ,
\end{gather*}
where $N\ge 3$ (for simplicity).
Note that the embeddings above are not compact.  
Consider the minimization problem:
for $2<p<\2$, 
\begin{gather*}
S_p\coloneqq \inf \l\{ \int_{\RN}|\nabla u(x)|^2+|u(x)|^2  \,\dd x ; \ 
u\in\Hrn, \ \int_{\RN}|u(x)|^p  \,\dd x  =1 \r\}, 
\end{gather*}
and take a minimizing sequence 
$(u_n)$. 
Roughly speaking, the concentration-compactness argument proceeds as follows:
firstly, find vectors $y_n\in\RN$ such that the  pulled back sequence $(u_n(\cdot+y_n))$ has a weakly and pointwise convergent subsequence (still denoted by the same numbering), say $u_n(\cdot+y_n) \to u_0$ (the existence  of such vectors is guaranteed since the minimizing sequence is not ``\emph{vanishing}"); 
secondly, the pulled back sequence $(u_n(\cdot+y_n))$ is also a minimizing sequence to $S_p$ ($2 < p < 2^*$);
finally, the Brezis-Lieb lemma and the concavity of $\R\ni t\mapsto t^{2/p}\in\R$ indicate that ``\emph{dichotomy}" does not occur and hence, $u_n(\cdot+y_n)$ strongly converges to $u_0$ in $\Hrn$. 
The above observation suggests that $u_n$ asymptotically behaves like $u_0(\cdot-y_n)$.  In our context, $u_0$ is called a \emph{profile} or a \emph{bubble}, and translations by vectors $(y_n)$ are called \emph{dislocations}, and finally $u_0(\cdot-y_n)$ is called a \emph{dislocated profile}. 
In general, dislocations consist of several group actions such as translations and dilations which are often isometric actions onto corresponding function spaces.
The argument based on the concentration-compactness principle is interpreted as follows.  
One seeks for a dislocated profile whose energy is 
a large portion of that of $u_n$. The existence of such a large profile is ensured if \emph{vanishing} does not occur. 
Such a possibility of dislocations is one of the main reasons 
of the lack of compactness of $(u_n)$, so the concentration-compactness theory 
is a precious method of asymptotic analysis of sequences with the lack of compactness.
In the case of the minimization problem 
\begin{equation}
S_{\2}\coloneqq \inf \l\{ \int_{\RN}|\nabla u(x)|^2  \,\dd x ; \ 
u\in\Hdrn, \ \int_{\RN}|u(x)|^{\2}  \,\dd x  =1 \r\},
\end{equation}
a similar argument is carried out to find a large bubble, and the concentration-compactness principle in~\cite{Lions 85 1, Lions 85 2} provides an alternative estimate of $L^{\2}$-norm, instead of the Brezis-Lieb lemma. 
A more detailed principle of concentration-compactness subordinated to 
the critical embedding
$\Hdrn \hookrightarrow L^{\2}(\RN)$
is described in, e.g., M.~Willem~\cite[Lemma~1.40]{W1}.  
Roughly speaking, 
the defect of compactness of a bounded sequence in $\Hdrn$ is analyzed 
by taking into account the following three ingredients: 
the weak limit; 
the energies of profiles translated to the space infinity;  
the energies of concentrated profiles. 
Then the energy or mass of the given bounded sequence is asymptotically 
decomposed into the above three parts of energies.
As is seen in the above reviews, Lions' ideas 
focus on energies or masses of bubbles instead of focusing on bubbles themselves.
Thus the classical principles of concentration-compactness can be regarded as a method of \emph{mass decomposition} or \emph{decomposition in measure} in contrast to \emph{profile decomposition}.

In~1984, M.~Struwe~\cite{Struwe} constructed a way of analyzing the defect of compactness especially of Palais-Smale sequences
subordinated to specific variational situations with no 
compactness conditions, and found that any bounded Palais-Smale sequence is asymptotically decomposed into a finite sum of dislocated profiles which are all weak solutions to suitable elliptic PDEs. 
Struwe method was a kickoff of profile decomposition 
in contrast to mass decomposition, and his work is well known as the \emph{global compactness} results. 
Several years later, Brezis and J.M.~Coron~\cite{BrezisCoron85} and A.~Bahri and Coron~\cite{BahriCoron88} established the global compactness results and they also revealed the so-called ``mutual orthogonality condition" or ``almost orthogonality condition". 

\paragraph{Profile decomposition}
From 90's, profile decomposition has been studied in more general situations not only in the variational situations:  
namely, ``for any'' bounded sequences in function spaces.
There have been many researches on profile decomposition, but we here pick up some of them and classify into two types in terms of residual terms: 
(i) profile decomposition with a single suffix residual term;  
(ii) profile decomposition with a double suffix residual term.  

\paragraph{Single-suffix profile decomposition}
S. Solimini~\cite{Solimini} studied the profile decomposition 
in the  homogeneous Sobolev space $\dot{W}^{1,p}(\RN)$  ($1<p<N$)  and it is described as follows: 
for any bounded sequence $(u_n)$ in $\dot{W}^{1,p}(\RN)$, 
there exist a family of functions $(\phi^l)$ in $\dot{W}^{1,p}(\RN)$, 
a family of real numbers $(\dsl{j}{l}{n})$ and a family of 
vectors $(\dsl{y}{l}{n})$ in $\RN$ such that, on a renumbered subsequence, 
\begin{align}\notag
&2^{-\dsl{j}{l}{n}\frac{N}{\p}} 
u_n \l( 2^{-\dsl{j}{l}{n}} \cdot + \dsl{y}{l}{n} \r) \to \phi^l\quad  \mbox{weakly in}\  \dot{W}^{1,p}(\RN) \\
&u_n=\sum_{l=1}^\infty 2^{\dsl{j}{l}{n}\frac{N}{\p}} 
\phi^l\l( 2^{\dsl{j}{l}{n}} (\cdot-\dsl{y}{l}{n} )\r)+r_n,
\end{align}
with  
\begin{equation}\notag
 \lim_{n\toinfty} \|r_n\|_{L^{p^* \!, \, q}(\RN)} =0, \qquad q>p,
\end{equation}
where $\p=pN/(N-p)$ and $L^{\p \!, \, q}(\RN)$ denotes 
the Lorentz space.
Moreover, (i) the asymptotic mutual orthogonality condition holds in the sense that 
for $k \neq l$, 
\begin{equation*}
|\dsl{j}{l}{n}-\dsl{j}{k}{n}| +
2^{\dsl{j}{k}{n}} |\dsl{y}{l}{n}-\dsl{y}{k}{n}| \toinfty
\quad \mbox{as} \  n\toinfty, 
\end{equation*}
and (ii) the decomposition of the Sobolev-energy (the so-to-speak energy estimate) is also gained:
\begin{align*}\eqntag\label{202202090010}
\limsup_{n\toinfty} \norm{u_n}^p_{\Wdrn} \ge 
\sum_{l=0}^\infty \norm{\bbl{\phi}{l}}^p_{\Wdrn}. 
\end{align*}
Those energy estimates will be well used in applications to PDEs or Calculus of Variations. For bounded Palais-Smale sequences to some semilinear elliptic functionals as used in Struwe's paper~\cite{Struwe}, each (nontrivial) critical point (that is expected to be a solution to the corresponding elliptic equations) has a priori least energy bounds, and hence, the above energy estimates implies that the number of nontrivial profiles must be finite. From this and the fact that each profile is a weak limit of suitably dislocated $u_n$, Struwe's results are re-established by means of the profile decomposition theory.

Solimini's theory has been
extended to more general and various settings, i.e., 
abstract Hilbert spaces, Banach spaces (including fractional Sobolev spaces, Besov spaces and Triebel-Lizorkin spaces), metric spaces and Sobolev spaces on Riemannian manifolds mainly by Solimini and C.~Tintarev (see also \cite{AdimurthiTintarev,DevillanovaSolimini,DevillanovaSoliminiTintarev,PalatucciPisante,SandeepShindlerTintarev,SchindlerTintarev,SoliminiTintarev,T-01,T-K} and references therein).

\paragraph{Double-suffix profile decomposition}
On the other hand, 
P.~G\'erard~\cite{Gerard} provided the profile decomposition in the homogeneous Sobolev space $\dot{H}^{s, 2}(\RN)$ ($s>0$) 
and S.~Jaffard~\cite{Jaffard} extended it to the homogeneous Sobolev space 
$\dot{H}^{s, p}(\RN)$ ($s>0,\  1<p<\infty$)
by the use of wavelet basis. 
Here, a brief review of Jaffard's results reads as follows: 
for any bounded sequence $(u_n)$ in $\dot{H}^{s,p}(\RN)$, 
there exist a family of functions $(\phi^l)$ in $\dot{H}^{s,p}(\RN)$, 
a family of real numbers $(\dsl{j}{l}{n})$ and a family of 
vectors $(\dsl{y}{l}{n})$ in $\RN$ such that, on a renumbered subsequence, 
\begin{equation}\notag
u_n=\sum_{l=1}^L 2^{\dsl{j}{l}{n}\frac{N-sp}{p}} 
\phi^l\l( 2^{\dsl{j}{l}{n}} (\cdot-\dsl{y}{l}{n} )\r)+r^L_n,
\end{equation}
with  
\begin{equation}\notag
\lim_{L\to \infty} \varlimsup_{n\toinfty} \|r^L_n\|_{L^q(\RN)} =0,
\end{equation}
where $q$ satisfies $s=N/p-N/q$.
Moreover, (i) the asymptotic mutual orthogonality condition holds in the sense that 
for $k \neq l$, 
\begin{equation*}
|\dsl{j}{l}{n}-\dsl{j}{k}{n}| +
2^{\dsl{j}{k}{n}} |\dsl{y}{l}{n}-\dsl{y}{k}{n}| \toinfty \quad 
\text{as}\ n\toinfty, 
\end{equation*}
and (ii) a decomposition in the Sobolev-energy is also gained:
\begin{align*}\eqntag\label{202202090020}
\limsup_{n\toinfty}\norm{u_n}^p_{\Hdsp}
\ge \sum_{l=0}^\infty \norm{\bbl{\phi}{l}}^p_{\Hdsp}. 
\end{align*}
The theories of G\'erard and Jaffard are extended by 
H.~Bahouri, A.~Cohen and G.~Koch~\cite{B-C-K} 
to more general  function spaces 
such as Besov spaces and Triebel-Lizorkin spaces by the use of wavelet basis.

The above decomposition theorem of Jaffard's type  is proved based on wavelet basis, 
so this proposition does not tell us relations 
between $u_n$ and each profile $\phi^l$,  
while  the Solimini-Tintarev-type decomposition 
states that each profile is given as the weak limit of 
reversely dislocated $u_n$. 
So the Solimini-Tintarev-type results seem to be extensions of Lions' ideas.

\paragraph{Relations between two-types of profile decomposition}
Devillanova-Solimini~\cite{DevillanovaSolimini} clearly exhibited the relations between the above two-types of profile decomposition. 
The double-suffix profile decomposition can be regarded as a \emph{weaker} form of the single-suffix one, but, as Devillanova-Solimini also states, dislocations $(\dsl{y}{l}{n},\dsl{j}{l}{n})$ in profile decomposition are not the same choice between the above two profile decompositions. 
In order to reach the \emph{infinite bulk} of profiles as in the first one, one should perform the so-called \emph{routing procedure} with subtle arguments as is discussed  in~\cite[Section~5]{DevillanovaSolimini}. 
Due to this routing procedure together with reselections of dislocations, Devillanova-Solimini suggested that double-suffix profile decompositions are converted to single-suffix ones. 

On the other hand, let us consider an \emph{inverse problem}:
given countably many profiles and dislocations, can we find a bounded sequence $(u_n)$ whose profile elements coincide with the given ones?
It is natural to consider $(u_n)$ defined by a single-suffix or double-suffix profile decomposition, but the single-suffix one solves the inverse problem as long as the dislocations are ``\emph{routed}'' as is observed in~\cite{DevillanovaSolimini}. 
However, the double-suffix one is always well-posed and the inverse problem is always solved by the double-suffix one.

\paragraph{Hybrid-type profile decomposition}
Regarding such subtle arguments on routing-procedure, 
N.~Ikoma and M.~Ishiwata discussed, in their personal  notes~\cite{Ikoma,Ishiwata} (in Japanese), a weaker formulation of Solimini-Tintarev-type profile decomposition theorems without the routing-procedure, and they obtained a double-suffix profile decomposition theorem in Hilbert spaces based on the \emph{weak-topological} arguments established by Solimini, Tintarev et al. 
This is a so-to-speak hybrid-type profile decomposition of Solimini-Tintarev-type one and G\'erard-Jaffard-type one. 
Meanwhile, T.~Tao~\cite[Theorem~4.5.3]{tao blog-book}
(in 2010 according to his blog) 
constructed a profile decomposition theorem 
in  Hilbert spaces, which is quite  
similar to those of Ikoma and Ishiwata. 
Such a profile decomposition theorem in general Hilbert spaces will be exhibited in \cref{theorem;general profile decomp.} in \cref{section;ProDecoHilbert} with its complete proof by the author for the sake of the reader's convenience.

We shall exhibit here
what the hybrid-type profile decomposition is like by applying \cref{theorem;general profile decomp.} 
to the case of the homogeneous Sobolev space $\Hdrn$, so that one sees the meaning of ``hybrid'' clearly.

\begin{corollary}[cf. \cref{theorem;general profile decomp.}]\label{cor11}
Let $(u_n)$ be a bounded sequence in $\dot{H}^1(\RN)$.
Then there exist a subsequence $(N(n)) \subset \NN \ (n \in \NN)$, 
$\bbl{w}{l} \in \dot{H}^1(\RN) \ (l \in\NN)$, 
$\dsl{y}{l}{N(n)} \in \RN \ (0 \le l \le n)$ and 
$\dsl{j}{l}{N(n)} \in \Z \ (0 \le l \le n)$ 
such that  the following hold.
\begin{enumerate}
\setlength{\itemsep}{0mm}

\item $\dsl{y}{0}{N(n)}=0,\  \dsl{j}{0}{N(n)}=0 \ (n \ge 0)$.
\item $|\dsl{j}{l}{N(n)}-\dsl{j}{k}{N(n)}| +
2^{\dsl{j}{l}{N(n)}} |\dsl{y}{l}{N(n)}-\dsl{y}{k}{N(n)}| \toinfty$
as $n \toinfty$ whenever $l\neq k$.
\item $2^{-\dsl{j}{l}{N(n)}\frac{N-2}{2}} 
u_{N(n)} \l( 2^{-\dsl{j}{l}{N(n)}} \cdot + \dsl{y}{l}{N(n)} \r) \to \bbl{w}{l}$ weakly in $\dot{H}^1(\RN)$.
\item Set the $L$-th residual term by 
\begin{equation*}
r^L_{N(n)}\coloneqq u_{N(n)}
-\sum_{l=0}^L 2^{\dsl{j}{l}{N(n)}\frac{N-2}{2}} 
\bbl{w}{l} \l( 2^{\dsl{j}{l}{N(n)}} (\cdot-\dsl{y}{l}{N(n)} )\r), \quad n \ge L \ge 0,  
\end{equation*}
and then, it holds that 
\begin{equation*}
2^{-\dsl{j}{k}{N(n)}\frac{N-2}{2}} 
r^L_{N(n)} \l( 2^{-\dsl{j}{k}{N(n)}} \cdot + \dsl{y}{k}{N(n)} \r)
\to 
\begin{cases}
0  &  \mbox{if} \ k=0,\ldots,L, \\ 
\bbl{w}{k}  &  \mbox{if}\ k \ge L+1,
\end{cases}
\end{equation*}
weakly in $\Hdrn$ as $n\toinfty$. 
\end{enumerate} 
Furthermore, it holds that 
\begin{equation}\label{202109100010}
\lim_{L\to \infty}\sup_{\phi\in U}\varlimsup_{n\toinfty}\sup_{y\in \RN \!, \,  j\in\Z}
\qty| \qty( \phi, 2^{-j\frac{N-2}{2}} 
r^L_{N(n)} ( 2^{-j} \cdot + y ) )_{\dot{H}(\RN)}|=0,
\end{equation}
where $U \coloneqq B_{\dot{H}^1(\RN)}(1)$,  which is 
the unit ball of $\dot{H}^1(\RN)$, that 
\begin{equation}\label{202109100020}
\varlimsup_{n\toinfty}\|u_{N(n)} \|_{\dot{H}^1(\RN)}^2
    =\sum_{l=0}^\infty \|\bbl{w}{l}\|_{\dot{H}^1(\RN)}^2 
    + \lim_{L\to \infty}\varlimsup_{n\toinfty}\|r^L_{N(n)} \|_{\dot{H}^1(\RN)}^2,
\end{equation}
and that 
\begin{equation}\label{202109100030}
\lim_{L\to\infty}\varlimsup_{n\to\infty}\|r^L_{N(n)}\|_{L^{\2}(\RN)}=0.
\end{equation}
\end{corollary}

There have been some attempts to construct profile decomposition theorems for Borel measures (see, e.g.,~\cite{Maris, tao blog-book}, and Maris obtained a theorem of profile decomposition for the Sobolev space $W^{1,p}(\RN)$ as a corollary of profile decomposition for Borel measures. However, we shall improve his results for Sobolev spaces in terms of the decomposition of the Sobolev norms. Meanwhile, note that our results below and Maris' one coincide with each other when $p=2$.

\medskip
According to Tao~\cite{tao blog-book}, double-suffix profile decomposition theorems are ``slightly more convenient'' for applications to dispersive and wave equations. Tao refers to~\cite{Killip} for such applications. 
However, the hybrid-type profile decomposition theorems, expanded in the present paper and~\cite{Okumura3}, will be more convenient in applications owing to the fact that each profile is obtained as a weak limit of suitably dislocated $u_n$. 
There are many attempts to apply double-suffix profile decomposition theorems to PDEs, and for details, we refer the reader to the following literature and references therein: 
\cite{GuoLiuWang17}
for nonlinear elliptic equations; 
\cite{IshiwataRIMS1,IshiwataRIMS2}
for nonlinear parabolic equations (see also~\cite{ChikamiIkedaTaniguchi}); 
\cite{BahouriGerard}
for nonlinear wave equations; 
\cite{KenigMerle,Keraani,Keraani2}
for nonlinear Schr\"odinger equations;
\cite{GallagherKochPlanchon,KenigKoch}
for Navier-Stokes equations.
Applications of Solimini-Tintarev-type profile decomposition theorems to PDEs and Calculus of Variations are shown in, e.g.,~\cite{B-C-K,PalatucciPisante,T-01,T-K} and references therein.

\subsection{Aim of the paper}
Our aim in the present paper is to extend the double-suffix profile decomposition theorem in general Hilbert spaces (\cref{theorem;general profile decomp.} in \cref{section;ProDecoHilbert}) to non-Hilbertian Sobolev spaces of inhomogeneous type, and it will also clarify the relations between a bounded sequence and its profiles via the weak convergence and an isometric group action. Those relations do not appear in the results of G\'erard and Jaffard.
Moreover, our results will be proved in a straightforward and simple manner without relying on the wavelet basis, so that those theorems might be readily extended to other function spaces with isometric group actions.
Such theorems may be well applied to PDEs and Calculus of Variations.
To this end, we shall firstly establish an abstract theory of profile decomposition with a double suffix residual term in general Banach spaces without giving detailed quantitative information 
such as~\eqref{202109100010}--\eqref{202109100030} in \cref{cor11}. As a corollary, the profile decomposition theorem with a double suffix residual term in general Hilbert spaces will be obtained. This abstract step also seems to be of independent interest as well. By the use of these abstract settings and from direct calculations of energies, we reach a profile decomposition theorem in Sobolev spaces. It is noteworthy that our energy estimates in the Sobolev norm is slightly sharper than those in other authors' papers, except for the Hilbertian case, i.e., $p=2$. 

Furthermore, we shall derive important results of decomposition of integral functionals, also regarded as the \emph{iterated Brezis-Lieb lemma}, subordinated to the profile decomposition in Sobolev spaces. 
From the viewpoints of profile decomposition, 
the classical Brezis-Lieb lemma roughly focuses 
on just two parts of $(u_n)$: 
the weak limit $u_0$ and all the other profiles $u_n-u_0$. 
On the contrary, our results of decomposition of integral functionals (the iterated Brezis-Lieb lemma) typically implies the decomposition of 
$\lim_{n\toinfty}\|u_n\|_{L^q(\RN)}^q$ into the sum of those of all profiles.
We shall also develop those decompositions of integral functionals in other Sobolev spaces where the residual term of profile decomposition is vanishing.

\smallskip
The author also wrote a paper~\cite{Okumura3} on the double-suffix profile decomposition in ``homogeneous" Sobolev spaces and decomposition of integral functionals in the same spirit as in this paper.

\subsection{Structure of the paper} 
The rest part of \cref{section;intro} contains the notation and settings employed throughout the paper. 
In \cref{section;AbsProDeco}, we shall discuss an abstract theory of profile decomposition in general reflexive Banach spaces and the behavior of the residual term. The results in the section will be a fundamental theorem on which profile decomposition in specific Banach spaces will be based. 
Also, a theorem of profile decomposition in general Hilbert spaces is exhibited. 
\cref{section;ProDecoInhomSobolev} is devoted to the profile decomposition of bounded sequences in inhomogeneous Sobolev spaces, which is our main interest in this paper. 
\cref{section;BLcubcrit} is concerned with decomposition of integral functionals (the iterated Brezis-Lieb lemma) subordinated to profile decomposition in inhomogeneous Sobolev spaces. 
Throughout the paper, results and proofs will be given in the same sections.

\subsection[Notation and settings]{Notation and functional analytic settings: $G$-weak convergence, dislocation space, $G$-complete continuity} 
As is reviewed in \cref{subsection;prologue}, 
analysis of the defect of compactness needs to 
chase dislocated profiles. Here several meanings of \emph{dislocation} appear in the context of profile decomposition: 
for instance, translations for  
inhomogeneous Sobolev spaces; translations and dilations for   
homogeneous Sobolev spaces. 
So it is necessary to establish a structure for chasing dislocated profiles, i.e.,  \emph{dislocation spaces} and \emph{$G$-weak convergence}  introduced in~\cite{T-K}; their attempts to set up suitable languages and structures for profile decomposition are of consideration so that we shall rely on them to establish our theorems.

First of all, we begin with 

\medskip
\noindent
{\bf Notation.}
Throughout the paper, we often use the following notation. 
\begin{enumerate}
\setlength{\itemsep}{0mm}

    \item We write $(n_j) \preceq (n)$ when the left-hand side is a subsequence of the right-hand side. 
    \item The set of integers greater than or equal to $l \in \Z$ is denoted by $\Z_{\ge l}$.
    \item For $\Lambda \in \N \cup \{0,+\infty\}$, the set of integers at least zero and at most  $\Lambda$ is denoted by $\NN^{<\Lambda+1}$, where we assume that $\Lambda +1 = +\infty$ when $\Lambda = +\infty$, i.e., 
    $\NN^{<\Lambda+1} = \{0, \ldots, \Lambda  \}$ if $\Lambda <+\infty$ and 
    $\NN^{<\Lambda+1} = \NN$ if $\Lambda = +\infty$. 
    Moreover, for real numbers $(a_L)_L$, we often write  $\lim_{L\to\Lambda} a_L = a_\Lambda$ whenever $\Lambda<\infty$. 
    \item We denote by $C$ a non-negative constant, which does not
depend on the elements of the corresponding space or set and may vary from line to line.
    \item For an exponent $p \in [1,\infty]$,  we denote by $p'$ the  H\"{o}lder conjugate exponent of $p$: $p'=p/(p-1)$ if $p\in]1,\infty[;$ $p'=\infty$ if $p=1;$ $p'=1$ if $p=\infty$. 
    \item For a normed space $X$,  we denote its norm by $\|\cdot\|$ or $\|\cdot\|_{X}$, 
    denote its dual space by $X^*$ and 
    denote the duality pairing between $X,X^*$ by $\la \cdot,\cdot\ra $ or $\la\cdot,\cdot\ra_X$.
    \item For a pre-Hilbert space $H$,  we denote its inner product by $(\cdot,\cdot)$ or 
    $(\cdot,\cdot)_H$.
    \item For a normed space $X$, we denote by $B_X(r)$ 
    the closed ball of radius $r>0$ centered at the origin, that is, 
    $B_X(r) = \qty{ u \in X; \ \|u\|_X \le r }$. 
    Meanwhile, a closed ball in $\RN$ of radius $r>0$ centered at $p\in \RN$ is denoted by $B(p,r)$, that is, 
    $B(p,r)\coloneqq \qty{x\in \RN; \ |x-p|\le r}$. 
    \item For a normed space $X$,  we denote by $\mathcal{B}(X)$ 
    the normed space of all bounded linear operators on $X$ equipped with the operator norm. 
    \item For a Banach space $X$ and for a
    bounded operator $T\in \mathcal{B}(X)$,  the adjoint operator
    of $T$ is denoted by $T^*$.
\end{enumerate}

From now on, we introduce some definitions and notions 
used in the framework of profile decomposition.
We recall two notions of convergence of bounded linear operators 
on a Banach space.

\begin{definition}[Convergence of operators]
Let $X$ be a Banach space.
For a sequence of bounded operators $(A_n)$ in $\mathcal{B}(X)$, 
$(A_n)$ is said to converge to $A\in \mathcal{B}(X)$
\emph{strongly} or 
\emph{operator-strongly}  if 
$A_n u \to Au$ strongly in $X$ 
for all $u\in X$ as $n\toinfty$,  and in this case, 
the strong convergence of operators is denoted by 
$A_n \to A$.
Also, $(A_n)$ is said to converge to $A\in \mathcal{B}(X)$
\emph{weakly} or \emph{operator-weakly} 
if $A_n u \to Au$ weakly in $X$ 
for all $u\in X$ as $n\toinfty$,  and in this case, 
the weak convergence of operators is denoted by 
$A_n \rightharpoonup A$. 
\end{definition}

In the framework of profile decomposition, a new notion of 
convergence named \emph{$G$-weak convergence} 
(an extension of the usual weak convergence) 
is established. It is intended that 
$G$-weak convergence to zero indicates that 
there are no longer nontrivial profiles. 
For more details, see~\cite{T-01}.

\begin{definition}[$G$-weak convergence]
\label{definition;G-weak convergence}
Let $X$ be a Banach space and let $G\subset \mathcal{B}(X)$
be a group of bijective isometries on $X$ (under operator composition).
A sequence $(u_n)$ in $X$ is said to be 
\emph{$G$-weakly convergent} to $u\in X$ provided that 
\begin{equation*}
\lim_{n\toinfty} \sup_{g\in G} \qty|\la \phi, g^{-1}(u_n-u) \ra |=0 
\mbox{\quad for all \ } \phi\in X^*,
\end{equation*}
or equivalently, provided that for any sequence $(g_n)$ in $G$,  
$g_n^{-1}(u_n-u) \to 0$ weakly in $X$ as $n\toinfty$.
\end{definition}

\begin{remark} \label{remark of G-weak conv}
\rm
\begin{enumerate}
    \item 
    Note that if $X$ is a Hilbert space, then linear bijective isometries 
    are also called unitary operators.   
    \item One can easily see that $G$-weak convergence always implies weak convergence by definition. 
    If $G=\{{\rm Id}_X\}$,  then the $G$-weak convergence is nothing 
    but the ordinary weak convergence in $X$. 
\end{enumerate}
\end{remark}

We then introduce suitable spaces for profile decomposition named \emph{dislocation spaces}. 
A general definition of them is exhibited in~\cite{T-01,T-K}, but for the sake of convenience for construction of a theory for Sobolev space cases, 
we here use a limited definition of dislocation spaces.

\begin{definition}[Dislocation space]\label{def;DislSp}
Let $X$ be a Banach space,  
and let $G \subset \mathcal{B}(X)$ be a \emph{group of bijective isometries} on $X$ (under operator composition). 
Then $G$ is called a \emph{dislocation group} if the following conditions are satisfied: 
\begin{enumerate}
\setlength{\itemsep}{0mm}

    \item For any sequence $(g_n)$ in $G$
    with $g_n\not\rightharpoonup 0$,  
    there exists a subsequence $(n_j) \preceq (n)$
    such that   
    $(g_{n_j})$ converges operator-strongly.
    \item For any sequence $(g_n)$ in $G$ with $g_n \not\rightharpoonup 0$ and for any sequence $(u_n)$ in $X$ with $u_n \to 0$ weakly, there exists a subsequence $(n_j) \preceq (n)$ such that $g_{n_j} u_{n_j} \to 0$ weakly in $X$. 
\end{enumerate}      
In that case, the pair $(X,G)$ is called a \emph{dislocation space}.
\end{definition}

\begin{lemma}\label{DslGrpRem}
Let $G$ be a group of bijective isometries on $X$ (under operator composition) and satisfy the following: 
for any sequence $(g_n)$ in $G$ with $g_n\not\rightharpoonup 0$, the adjoint $(g_n^*)$ has an operator-strongly convergent subsequence in $\mathcal{B}(X^*)$. 
Then $G$ satisfies the assertion~(ii) of \cref{def;DislSp}.
\end{lemma}

\begin{proof}
Assume that up to a subsequence, $(g_n^*)$ converges operator-strongly.
Then for any $u_{n} \in X$ with $u_{n}\to 0$ weakly in $X$ and for any $\phi \in X^*$, one sees that 
\begin{equation*}
    \la  \phi, g_{n} u_{n} \ra  = \la  g_{n}^* \phi, u_{n}\ra  = o(1)
    \quad (n\to\infty),
\end{equation*}
which completes the proof.
\end{proof}

From the above definition, one can immediately show the following corollary. 

\begin{corollary}\label{cor;dsl sp equiv norm}
Let $X$ be a linear space equipped with two different norms $\|\cdot\|_1$ and $ \|\cdot\|_2$ which are equivalent and complete, and let $G\subset\mathcal{B}(X)$ be a group (under operator composition) of bijective isometries with respect to both $\|\cdot\|_1$ and $\|\cdot\|_2$. 
If $((X, \|\cdot\|_1), G)$ is a dislocation space, then $((X, \|\cdot\|_2), G)$ is also a dislocation space.  
\end{corollary}

In applications of profile decomposition, dislocation spaces are often function spaces embedded into other function spaces and those embeddings have a property of $G$-complete continuity defined as follows.

\begin{definition}[$G$-complete continuity] \label{definition;G-compl conti} 
Let $(X,G)$ be a dislocation space and let $Y$ be a 
normed space. A continuous linear operator $T:X \to Y$
is said to be \emph{$G$-completely continuous} if 
for any sequence $(u_n)$ in X, 
$Tu_n \to Tu$ strongly in $Y$ whenever 
$u_n \to u$ $G$-weakly in $X$.
\end{definition}

\begin{remark}
{\rm
The property of 
$G$-complete continuity is also called 
$G$-cocompactness~(see \cite{T-K} for instance).
}
\end{remark}

Examples of $G$-completely continuous embeddings will be introduced in \cref{lemma;G-complete continuity of Wmp} below. See also~\cite{T-01, T-K} for other $G$-completely continuous embeddings.

\section[Abstract theory of PD]{Abstract theory of profile decomposition}\label{section;AbsProDeco}
This section broadly consists of two significant results; one of which is an abstract theorem of profile decomposition and another concerns relations between profile decomposition and $G$-completely continuous embeddings.

\subsection{Fundamental theorem of profile decomposition}
We intend to show that in dislocation spaces, any bounded sequence has a profile decomposition (without energy estimates).
When applied to specific dislocation spaces such as Sobolev spaces, energy decompositions for Sobolev norms or Lebesgue norms will be obtained with the help of direct calculations of those norms.

\begin{theorem}\label{theorem;ProDeco in X}
Let $(X,\|\cdot\|_X)$ be a reflexive Banach space, let 
$(X, G)$ be a dislocation space, and let $(u_n)$ be a bounded sequence in $X$. 
Then there exist a number $\Lambda \in \N\cup\{0,+\infty\}$, a subsequence $(N(n)) \subset \NN \ (n \in \NN)$, 
profiles 
$\bbl{w}{l} \in X \ (l \in \NN^{<\Lambda+1})$, 
dislocations
$\dsl{g}{l}{N(n)} \in G \ (l \in \NN^{<\Lambda+1}, \ n \in \Z_{\ge l})$, 
and residual terms $r^L_{N(n)}\in X$ $(L\in\NN^{<\Lambda+1},\ n\in\Z_{\ge L})$ with the relation of a double-suffix profile decomposition 
\begin{align*}
u_{N(n)} = \sum_{l=0}^L \dsl{g}{l}{N(n)}\bbl{w}{l} + r^L_{N(n)}, \quad L\in\NN^{<\Lambda+1}, \ n\in \Z_{\ge L}, 
\end{align*}
such that the following hold.
\begin{enumerate}
\setlength{\itemsep}{0mm}
\setlength{\parskip}{0mm}

\item $\dsl{g}{0}{N(n)}={\rm Id}_X \ (n \ge 0), \quad \bbl{w}{l}\neq 0 \ (1 \le l \in \NN^{<\Lambda+1}).$
\item $\inv{\dsl{g}{l}{N(n)}}\dsl{g}{m}{N(n)} \to 0$ operator-weakly 
as $n \toinfty$ whenever $l\neq m \in \NN^{<\Lambda +1}.$
\item $\inv{\dsl{g}{l}{N(n)}}u_{N(n)} \to \bbl{w}{l}$ weakly in $X$ 
$(n\toinfty, \ l\in \NN^{<\Lambda +1}).$
\item For $k\in\NN^{<\Lambda+1}$, 
\begin{equation}
\label{eq;ProDeco in X 1}
\inv{\dsl{g}{k}{N(n)}} r^{L}_{N(n)} \to
\begin{cases}
0  &  \mbox{if}\  k=0,\ldots,L, \\ 
\bbl{w}{k}  & \mbox{if}\  k\ge L+1, 
\end{cases}
\end{equation}
weakly in $X$ as $n\toinfty$.
\end{enumerate} 
Furthermore, if either $\Lambda =\infty$ and $\|\bbl{w}{l}\|_X \to 0$ as $l\to \Lambda$, or else $\Lambda<\infty$, then 
\begin{equation}
\label{eq;Ishiwata condition X}
\lim_{L\to \Lambda}\sup_{\phi\in U}\varlimsup_{n\toinfty}\sup_{g\in G}
\qty| \la  \phi, g^{-1}r^L_{N(n)} \ra_X |=0,
\end{equation}
where $U \coloneqq B_{X^*}(1)$.   
\end{theorem}

\begin{remark}
\begin{itemize}

\item 
The above theorem gives \emph{qualitative} assertions of profile decomposition, and \emph{quantitative} one~\eqref{eq;Ishiwata condition X} is obtained under a further assumption that either $\Lambda =\infty$ and $\|\bbl{w}{l}\|_X \to 0$ as $l\to \Lambda$, or else $\Lambda<\infty$. 
However, this further assumption will be ensured in theorems of profile decomposition below in the present paper and in~\cite{Okumura3}, by virtue of the direct calculations for the decompositions in energy like~\eqref{202202090010} and~\eqref{202202090020}, which show that $\norm{\bbl{w}{l}}_X\to 0$ as $l\toinfty$ if $\Lambda=\infty$. 

\item 
As for the assumption ``either $\Lambda =\infty$ and $\|\bbl{w}{l}\|_X \to 0$ as $l\to \Lambda$, or else $\Lambda<\infty$'', 
Solimini-Tintarev~\cite{SoliminiTintarev} generally verified this further assumption, by reformulating the profile decomposition theory by means of the so-called ``$\Delta$-convergence''. 
They established a $\Delta$-convergence-version of the profile decomposition theory for uniformly convex and uniformly smooth Banach spaces, where the above further assumption always holds true, and they also showed that if the Banach space satisfies Opial's condition, the $\Delta$-limits coincide with the weak limits, so that the above further assumption are also true with respect to the weak-topological profile decomposition theory. 
It is noteworthy that Hilbert spaces, Besov spaces and Triebel-Lizorkin spaces (including Sobolev spaces) enjoys Opial's condition, so that our assumption is satisfied in those cases.

\end{itemize}
\end{remark}

We make the following notation. 

\begin{definition}
A bounded sequence $(u_n)$ in $X$ is said to have 
a profile decomposition on a subsequence $(N(n))$ with  
profile elements $(\bbl{w}{l}, \dsl{g}{l}{N(n)}, \Lambda) \in X \times G\times(\N\cup\{0,+\infty\})$ $( l \in  \NN^{<\Lambda +1}, \ n \in \Z_{\ge l} )$ if $\Lambda \in \N\cup \{0,+\infty\}$, $N(n) \in \NN$, $\bbl{w}{l} \in X$ and $\dsl{g}{l}{N(n)} \in G$ satisfy all conditions as in \cref{theorem;ProDeco in X}. In that case, $\bbl{w}{l}$ is called the \emph{$l$-th profile}, and $(\dsl{g}{l}{N(n)})$ is called the \emph{$l$-th dislocations}. Especially, the $0$-th profile $\bbl{w}{0}$ is called the \emph{principal profile}, which is given as the weak limit of $(u_{N(n)})$. 
Moreover, the profile decomposition of $(u_n)$ is said to be \emph{exact} if the condition~\eqref{eq;Ishiwata condition X} is satisfied, and thus, this condition is called the \emph{exactness condition} of profile decomposition. 
\end{definition}

\begin{remark} 
\rm 
\begin{enumerate}
\item The condition $\dsl{g}{l}{N(n)}\in G \ ( l \in \NN^{<\Lambda +1}, \ n \in \Z_{\ge l})$ means that for fixed $l \in \NN^{<\Lambda +1}$, the $l$-th dislocations $(\dsl{g}{l}{N(n)})$ are  well-defined provided $n \ge l$. On the other hand, $\dsl{g}{l}{N(n)}$ is not necessarily well-defined if $n<l$. Although this restriction might seem strange, it is attributed to the technical difficulty of the diagonal argument (see also the proof of \cref{theorem;ProDeco in X} below). 
By the way, one can redefine $\dsl{g}{l}{N(n)}=\mbox{Id}_X$ for $n<l$ for the sake of well-definedness, but that does not make any improvement. 

\item Regarding the assertion~(i), the principal profile $\bbl{w}{0}$ is possibly zero. This is because the principal profile is defined as the weak limit of $(u_{N(n)})$. 
\end{enumerate}
\end{remark} 

Once the number of nontrivial profiles $\Lambda$ turns out to be finite, one can readily get a simpler version of the above theorem.

\begin{corollary}
Suppose that the same assumptions as in \cref{theorem;ProDeco in X} are satisfied. 
In addition, assume that $\Lambda < +\infty$, i.e., 
the number of nontrivial profiles is finite. 
Then regarding the final residual term given by 
\begin{equation*}
r^\Lambda_{N(n)} = u_{N(n)} -\sum_{l=0}^\Lambda \dsl{g}{l}{N(n)} \bbl{w}{l}, 
    \s n \ge \Lambda, 
\end{equation*}
the relation~\eqref{eq;Ishiwata condition X} turns to 
\begin{equation}
\lim_{n\toinfty} \sup_{g \in G} 
\qty| \la  \phi, g^{-1} r^\Lambda_{N(n)} \ra_X |=0 
\s \mbox{for all}\s \phi \in U,  \label{residue G weak conv X}
\end{equation}
which means that 
the final residual term is $G$-weakly convergent to zero in $X$ (cf. \cref{definition;G-weak convergence}). 
\end{corollary}

In the case of the above corollary, we see that 
the residual term strictly converges to zero $G$-weakly in $X$, which ensures that there remain no longer nontrivial profiles. Thus the conditions~\eqref{eq;Ishiwata condition X} and~\eqref{residue G weak conv X} may indicate that the residual term is shrinking as we subtract iteratively more and more nontrivial dislocated profiles, and we can finally take up all nontrivial profiles of $(u_n)$. Therefore, roughly speaking, the conditions~\eqref{eq;Ishiwata condition X} and~\eqref{residue G weak conv X} imply the completion of performing the profile decomposition, and so this is why we shall call the condition~\eqref{eq;Ishiwata condition X} the exactness condition of profile decomposition.

\subsection{Behavior of the residual term}
Now, let $(X,G)$ be as in \cref{theorem;ProDeco in X} and let $(Y,\|\cdot\|_Y)$ be a normed space.
Suppose that the embedding 
$X\hookrightarrow Y$ 
is $G$-completely continuous.
%
We next investigate relations between 
the exactness condition~\eqref{eq;Ishiwata condition X} 
and $G$-completely continuous embeddings. 
When one gets~\eqref{eq;Ishiwata condition X}, then it is natural to discuss the behavior of the residual term in $Y$. 
The following general result implies that 
the residual term satisfying~\eqref{eq;Ishiwata condition X} may become arbitrarily small in $Y$.
Hence this results may be considered as a ``weaker form of $G$-complete continuity". 

\begin{theorem}\label{theorem;weak G-comp conti}
Let $(X,G)$ be as in \cref{theorem;ProDeco in X} and let $(Y,\|\cdot\|_Y)$ be a normed space.
Suppose that the embedding 
$X\hookrightarrow Y$
is $G$-completely continuous.
Also assume that a double-suffix sequence $(u^L_n)$ in $X$ satisfies that 
\begin{align*}
&\sup_{n,L\in\N}\|u^L_n\|_X<\infty, \\
&\lim_{L\to\infty}\sup_{\phi\in B_{X^*}(1)}
\varlimsup_{n\to\infty}\sup_{g\in G}
\qty|\la \phi,g^{-1}u^L_n\ra_X |=0.
\end{align*}
Then it holds that 
\begin{equation*}
\lim_{L\to\infty}\varlimsup_{n\to\infty}
\|u^L_n\|_Y=0.
\end{equation*}
\end{theorem}

\begin{proof}
Suppose on the contrary that there is $\varepsilon >0$ such that 
$(u^L_n)$ satisfies 
\begin{align*}
&\varlimsup_{L\to\infty}\varlimsup_{n\to\infty} \|u^L_n\|_Y \ge \varepsilon.
\end{align*}
Fix $\phi\in X^*$.
Then for all $j\in\N$ there exists $L_j\in\N$ such that 
\begin{align*}
&\varlimsup_{n\to\infty}
\sup_{g\in G} \qty|\la  \phi, g^{-1} u^{L_j}_n \ra_X  | \le 2^{-j}, \\ 
&\varlimsup_{n\to\infty} \|u^{L_j}_n\|_Y \ge \varepsilon/2.
\end{align*}
Furthermore, there exists $n_j\in\N$ such that 
\begin{align*}
&\sup_{g\in G} \qty|\la \phi, g^{-1} u^{L_j}_{n_j} \ra_X  | \le 2^{1-j}, \\ 
&\|u^{L_j}_{n_j}\|_Y \ge \varepsilon/4.
\end{align*}
Thus one finds that 
\begin{equation*}
\lim_{j\to\infty} \sup_{g\in G} \qty|\la \phi, g^{-1} u^{L_j}_{n_j} \ra_X  | =0,
\end{equation*}
which means that $u^{L_j}_{n_j}\to 0$ $G$-weakly in $X$. 
Hence the $G$-completely continuous embedding $X\hookrightarrow Y$ yields 
\begin{equation*}
\lim_{j\to\infty}\|u^{L_j}_{n_j}\|_Y=0.
\end{equation*}
But this contradicts $\|u^{L_j}_{n_j}\|_Y \ge \varepsilon/4$ for all $j\in\N$.
Thus the proof is complete. 
\end{proof}

\subsection{Profile decomposition in general Hilbert spaces} \label{section;ProDecoHilbert}
We shall provide a theorem of profile decomposition with a double suffix residual term in general Hilbert spaces offered by Ikoma, Ishiwata and Tao, independently, and its complete proof for the sake of convenience. 
In Hilbert cases, there is a simple sufficient condition for a Hilbert space to be a dislocation space.

\begin{proposition}[Hilbert spaces as dislocation spaces]\label{prop;Hilbert dislocation space}
Let $H$ be a Hilbert space and let 
$G$ be a subgroup of the unitary group on $H$ satisfying 
the following condition:
every sequence $(g_n)$ in $G$ with $g_n\not\rightharpoonup 0$ has an operator-strongly convergent subsequence.
Then $G$ is a dislocation group and $(H,G)$ is a dislocation space.
\end{proposition}

\begin{proof}
It is well known that unitary operators are isometric bijections and that the inverse operator of a unitary operator $T$ is equal to the adjoint operator of $T$.
Now let $(g_n)$ be a sequence in $G$ such that $g_n \not \rightharpoonup 0$, which is equivalent to $g_n^*\not\rightharpoonup 0$ in $\mathcal{B}(H)$. 
Then, by assumption, $(g_n)$ and $(g_n^*)$ have operator-strongly convergent subsequences.
Thus \cref{DslGrpRem} implies~(ii) of \cref{def;DislSp}, hence the conclusion.
\end{proof}

\begin{remark}
\rm 
Concrete examples of dislocation Hilbert spaces are exhibited in, e.g., \cite{T-K}. 
\end{remark}

Now we reach the profile decomposition theorem in Hilbert spaces, which the author learned from~\cite{Ikoma,Ishiwata,tao blog-book}.

\begin{theorem}\label{theorem;general profile decomp.}
Let $(H,G)$ be a dislocation Hilbert space and let 
$(u_n)$ be a bounded sequence in $H$.
Then there exist a number $\Lambda \in \N\cup \{0,+\infty\}$, a subsequence $(N(n)) \subset \NN \ (n \in \NN)$, 
profiles 
$\bbl{w}{l} \in H \ ( l \in \NN^{<\Lambda+1})$, 
dislocations
$\dsl{g}{l}{N(n)} \in G \ (l \in \NN^{<\Lambda+1}, \  n \in \Z_{ \ge l})$,  
and residual terms $r^L_{N(n)}\in H$ $(L\in\NN^{<\Lambda+1}, \ n\in\Z_{\ge L})$ with the relation of a double-suffix profile decomposition 
\begin{align*}
u_{N(n)} = \sum_{l=0}^L \dsl{g}{l}{N(n)}\bbl{w}{l} + r^L_{N(n)}, \quad L\in\NN^{<\Lambda+1}, \ n\in\Z_{\ge L}, 
\end{align*}
such that the following hold:
\begin{enumerate}
\setlength{\itemsep}{0mm}
\setlength{\parskip}{0mm}

\item $\dsl{g}{0}{N(n)}={\rm Id}_H \ (n \ge 0), \quad \bbl{w}{l}\neq 0 \ (1 \le l \in \NN^{<\Lambda +1}).$
\item $\inv{\dsl{g}{l}{N(n)}}\dsl{g}{m}{N(n)} \rightharpoonup 0$
as $n \toinfty$ whenever 
$l\neq m \in \NN^{<\Lambda +1}.$
\item $\inv{\dsl{g}{l}{N(n)}}u_{N(n)} \to \bbl{w}{l}$ weakly in $H \ (n \toinfty, \ l \in \NN^{<\Lambda +1}).$
\item 
For $k\in\NN^{<\Lambda+1}$, 
\begin{equation*}
\inv{\dsl{g}{k}{N(n)}} r^{L}_{N(n)} \to 
\begin{cases}
0  &  \mbox{if} \ k=0,\ldots,L, \\
\bbl{w}{k}  &  \mbox{if} \ k\ge L+1, 
\end{cases}
\end{equation*}
weakly in $H$ as $n\toinfty$. 
\end{enumerate} 
Furthermore, it holds that 
\begin{equation*}
\varlimsup_{n\toinfty}\|u_{N(n)} \|^2
    =\sum_{l=0}^\Lambda \|\bbl{w}{l}\|^2 
    + \lim_{L\to\Lambda}\varlimsup_{n\toinfty}\|r^L_{N(n)} \|^2,
\end{equation*}
and that 
\begin{equation}
\label{eq;Ishiwata condition H}
\lim_{L\to \Lambda}\sup_{\phi\in U}\varlimsup_{n\toinfty}\sup_{g\in G}
\qty|\qty( \phi, g^{-1}r^L_{N(n)} )_H|=0,
\end{equation}
where $U \coloneqq B_{H}(1)$. 
Especially, there holds
\begin{equation*} 
\varlimsup_{n\toinfty}\|u_{N(n)} \|^2 \ge\sum_{l=0}^\Lambda \|\bbl{w}{l}\|^2.
\end{equation*}
\end{theorem}

An essential difference between the Tintarev-type profile decomposition and the above one is the exactness  condition~\eqref{eq;Ishiwata condition H}.

\begin{proof}
\cref{theorem;ProDeco in X} implies the existence of profile elements $(\bbl{w}{l}, \dsl{g}{l}{N(n)},\Lambda )\in H \times G\times (\N\cup\{0,\infty\})$ $( l \in \NN^{<\Lambda+1}, \ n \in \Z_{ \ge l})$ such that 
\begin{alignat}{2}
&\notag \dsl{g}{0}{N(n)}={\rm Id}_H &\quad &(n \ge 0), \\
&\notag \bbl{w}{l}\neq 0 &\quad &(l\in \NN^{<\Lambda+1}), \\
&\label{202103250060} 
\inv{\dsl{g}{l}{N(n)}}\dsl{g}{m}{N(n)} \rightharpoonup 0
&\quad &\mbox{as} \ n \toinfty \  \mbox{whenever} \  l\neq m, \\
&
\inv{\dsl{g}{l}{N(n)}}u_{N(n)} \to \bbl{w}{l} &\quad &\mbox{weakly in} \  H,\notag \\
&
r^L_{N(n)}\coloneqq u_{N(n)}-\sum_{l=0}^L \dsl{g}{l}{N(n)}\bbl{w}{l} &\quad &(L \in \NN^{<\Lambda +1}, \ n \in \Z_{\ge L}), \notag \\ 
&\label{202103250111}
\inv{\dsl{g}{l}{N(n)}} r^{L}_{N(n)} \to 
\begin{cases}
0,  &   0 \le l \le L, \\
\bbl{w}{l},  &  l\ge L+1, 
\end{cases}
 &\quad &\mbox{weakly in} \ H \ (0\le l \le L, \ n\toinfty).  
\end{alignat}

\paragraph{Energy decomposition}
We shall show the energy estimate. 

\begin{claim}
It holds that 
\begin{equation}
\label{eq;100200}
\varlimsup_{n\toinfty}\|u_{N(n)}\|_{H}^2 
= 
\sum_{l=0}^\Lambda \|\bbl{w}{l}\|_H^2
+\lim_{L\to \Lambda} \varlimsup_{n\toinfty}\|r^L_{N(n)}\|_H^2.
\end{equation}
Furthermore, it follows that 
\begin{equation}
\label{eq;1002001}
\lim_{L\to \Lambda}\sup_{\phi\in B_H(1)}\varlimsup_{n\toinfty}\sup_{g\in G}
\qty| \qty( \phi, g^{-1}r^L_{N(n)} )_H|=0.
\end{equation}
\end{claim}

\begin{proof}[Proof of Claim]
For $L\in \NN^{<\Lambda +1}$, it follows that 
\begin{align}
\|u_{N(n)}\|_H^2 &=
\l( \sum_{l=0}^L\dsl{g}{l}{N(n)}\bbl{w}{l} +r^L_{N(n)},
\sum_{k=0}^L\dsl{g}{k}{N(n)}\bbl{w}{k} +r^L_{N(n)}
\r)_H \label{eq;100201}\\
&\notag 
=\sum_{l=0}^L\sum_{k=0}^L 
\qty(\dsl{g}{l}{N(n)}\bbl{w}{l},\dsl{g}{k}{N(n)}\bbl{w}{k})_H
\\
&\notag 
\qquad \qquad 
+2\sum_{l=0}^L \qty(\dsl{g}{l}{N(n)}\bbl{w}{l},r^L_{N(n)})_H
+\|r^L_{N(n)}\|_H^2. 
\end{align}
Also, by \eqref{202103250060} and since  $g\in G$ is unitary (see also \cref{remark of G-weak conv}), one gets 
\begin{align}
\label{eq;100202}
\sum_{l=0}^L\sum_{k=0}^L 
\qty(\dsl{g}{l}{N(n)}\bbl{w}{l},\dsl{g}{k}{N(n)}\bbl{w}{k})_H
&=\sum_{l=0}^L\sum_{k=0}^L 
\qty(\bbl{w}{l},\inv{\dsl{g}{l}{N(n)}}\dsl{g}{k}{N(n)}\bbl{w}{k})_H \\
&\notag =\sum_{l=0}^L\sum_{k=0}^L 
\qty(\bbl{w}{l},\delta_{lk}\bbl{w}{k})_H + o(1) \\
&\notag =\sum_{l=0}^L \|\bbl{w}{l}\|_H^2 +o(1)
\end{align}
as $n\toinfty$.
In addition, from~\eqref{202103250111}, 
it follows that 
\begin{align}\label{eq;100203}
\sum_{l=0}^L \qty(\dsl{g}{l}{N(n)}\bbl{w}{l}, r^L_{N(n)} )_H 
=
\sum_{l=0}^L \qty(\bbl{w}{l}, \inv{\dsl{g}{l}{N(n)}}r^L_{N(n)} )_H 
=o(1)
\end{align}
as $n\toinfty$.
Combining \eqref{eq;100201}--\eqref{eq;100203}
and passing to the limit as $n\toinfty$,  
one obtains  
\begin{equation*}
\varlimsup_{n\toinfty}\|u_{N(n)}\|_{H}^2 
= 
\sum_{l=0}^L \|\bbl{w}{l}\|_H^2
+\varlimsup_{n\toinfty}\|r^L_{N(n)}\|_H^2.
\end{equation*}
Passing to the limit as $L\to \Lambda$, we get~\eqref{eq;100200}. 
Thus~\eqref{eq;1002001} follows from \cref{theorem;ProDeco in X}.  
\end{proof}

Now the proof has been complete.
\end{proof}

\subsection{Proof of \cref{theorem;ProDeco in X}}
The proof of \cref{theorem;ProDeco in X} below consists of iterative procedures to obtain and chase profiles. 
In this subsection, we fix a reflexive Banach space $(X, \|\cdot\|_X)$, and $(X,G)$ is supposed to be a dislocation space with a dislocation group $G$.

\paragraph{Key lemmas}
Firstly, we provide plain lemmas which are key steps of the proof. Those lemmas are established owing to~\cite[Section~4.3]{T-01} with some modifications.  
The modifications result from our limited definitions of dislocation spaces. 
So we also give proofs for them. 

\begin{lemma}\label{Lemma;T-01_431}
Let $(g_n)$ be a sequence in $G$. 
If $g_n \rightharpoonup 0$,  then 
$\inv{g_n} \rightharpoonup 0$ as $n\toinfty$. 
\end{lemma}

\begin{proof}
Let $(g_n)$ be a sequence in $G$ such that $g_n \rightharpoonup 0$.
Suppose $g_n^{-1} \not\rightharpoonup 0$ as $n\toinfty$.
By the definition of dislocation groups, $(\inv{g_n})$ converges to some $g\in\mathcal{B}(X)$ operator-strongly on a certain subsequence.
On this subsequence, and for any $u\in X\setminus\{0\}$,
one sees that 
\begin{equation*}
g_n(\inv{g_n}u)=g_n(\inv{g_n}u-gu)+g_n(gu)
\to 0\ \mbox{weakly in}\ X.
\end{equation*}
However, $g_n(\inv{g_n}u)=u$ and thus one has $u=0$. 
This is a contradiction, and hence one can conclude that $\inv{g_n}\rightharpoonup 0$.
\end{proof}

\begin{lemma}\label{Lemma;T-01_433}
Let $(u_n)$ be a bounded sequence in $X$.
If two sequences $(\dsl{g}{1}{n})$ and $(\dsl{g}{2}{n})$
in $G$ satisfy $\dsl{g}{1}{n}^{-1}u_n \to \bbl{w}{1}$ weakly in $X$ and $\dsl{g}{2}{n}^{-1}(u_n-\dsl{g}{1}{n}\bbl{w}{1})\to  \bbl{w}{2}\neq 0$ weakly in $X$, 
then $\dsl{g}{1}{n}^{-1}\dsl{g}{2}{n}\rightharpoonup 0$ as $n\toinfty$.
\end{lemma}

\begin{proof}
Let $(u_n), (\dsl{g}{j}{n}), \  j=1,2,$ and $\bbl{w}{j}, \ j=1,2,$ be 
as in \cref{Lemma;T-01_433}.
Suppose $\dsl{g}{1}{n}^{-1}\dsl{g}{2}{n} \not\rightharpoonup 0$ as $n\toinfty$.
By the definition of dislocation groups, 
$\dsl{g}{1}{n}^{-1}\dsl{g}{2}{n}$ 
has an operator-strongly convergent subsequence (renumbered), say 
$\dsl{g}{1}{n}^{-1}\dsl{g}{2}{n} \to g \in \mathcal{B}(X)$ 
as $n\toinfty$.
It follows from 
the definition of dislocation groups that 
up to a subsequence,  
\begin{equation*}
\dsl{g}{1}{n}^{-1}\dsl{g}{2}{n} 
\l( 
\dsl{g}{2}{n}^{-1} (u_n - \dsl{g}{1}{n} \bbl{w}{1}) - \bbl{w}{2} 
\r) \to 0\quad \mbox{weakly in}\ X,
\end{equation*}
yielding 
\begin{equation*}
    \dsl{g}{1}{n}^{-1}(u_n - \dsl{g}{1}{n} \bbl{w}{1}) \to g\bbl{w}{2}.
\end{equation*}
On the other hand, we have 
\begin{equation*}
    \dsl{g}{1}{n}^{-1}(u_n - \dsl{g}{1}{n} \bbl{w}{1}) \to 0\quad \mbox{weakly in}\ X,
\end{equation*}
so that 
\begin{equation*}
    g\bbl{w}{2} = 0,
\end{equation*}
which is a contradiction since 
$g$ must be an isometry and $\bbl{w}{2}\neq 0$.
\end{proof}

\begin{lemma}\label{Lemma;T-01_434}
Let $(u_n)$ be a bounded sequence in $X$ and let 
sequences $(\dsl{g}{l}{n})_n$ in $G$ and $\bbl{w}{l}\in X$ 
$(l=0,\ldots,L)$ satisfy 
$\dsl{g}{0}{n}={\rm Id}_X, \ \dsl{g}{l}{n}^{-1}u_n \to \bbl{w}{l}$ weakly in $X$ $(l=0,\ldots,L)$ 
and $\dsl{g}{l}{n}^{-1}\dsl{g}{m}{n} \rightharpoonup 0$ 
whenever $0\le l<m \le L$ as $n\toinfty$.
Assume that there exists a sequence $(\dsl{g}{L+1}{n})$ in $G$
such that, on a renumbered subsequence, 
$$
\dsl{g}{L+1}{n}^{-1}\l(
u_n -\sum_{l=0}^L \dsl{g}{l}{n}\bbl{w}{l}    
\r) \to \bbl{w}{L+1}\neq 0\quad \mbox{weakly in}\ X.
$$ 
Then for $l=0,\ldots,L$, 
$$
\dsl{g}{l}{n}^{-1}\dsl{g}{L+1}{n} \rightharpoonup 0
\quad \mbox{as}\s n\toinfty.
$$
\end{lemma}

\begin{proof}
\cref{Lemma;T-01_434} follows from induction on $L$  
with the definition of dislocation groups and \cref{Lemma;T-01_433}.
\end{proof}

\paragraph{Main body}
With the above lemmas, we shall prove \cref{theorem;ProDeco in X} based on~\cite{Ikoma,Ishiwata,tao blog-book,T-01,T-K}.

\begin{proof}[Proof of \cref{theorem;ProDeco in X}]
We divide the proof into six steps. 

\begin{settings}
For a bounded sequence $(u_n)$ in $X$,  we define 
\begin{align}\notag
D[(u_n)] &\coloneqq \bigl\{ w\in X; \ \exists\,(n_k) \preceq (n), \ 
\exists\, (g_{n_k}) \subset G \\
&\hspace{4cm} \mbox{s.t.} \ g_{n_k}^{-1} u_{n_k} \to w 
\ \mbox{weakly in} \ X \bigr\}, \\
p[(u_n)] &\coloneqq \sup \qty{ \|w\|_{X}; \ w \in D[(u_n)] }.
\end{align}
Then one can check that 
\begin{align}\notag
p[(u_n)] = 0 \quad 
&\Leftrightarrow \quad D[(u_n)]=\{0\} \\
&\Leftrightarrow  \quad 
u_n \to 0 \quad \mbox{$G$-weakly in} \  X  \s \mbox{as\s }  n \toinfty.
\end{align}
Moreover, since a dislocation group consists of isometries and each weakly convergent sequence is bounded, one sees that $p[(u_n)]$ is finite and bounded by $\sup_{n}\|u_n\|_{X}$. 
Note that $D[(u_n)]$ stands for the set of all possible profiles of $(u_n)$. 
\end{settings}

\begin{step}[Base step]
Let $(u_n)$ be a bounded sequence in $X$. 
From the reflexivity of $X$, there exist a subsequence $(i(0,n)) \preceq (n)$, $n\ge 0$,   
and $\bbl{w}{0} \in X$ such that 
\begin{equation}\notag 
u_{i(0,n)} \to \bbl{w}{0} \quad \mbox{weakly}.
\end{equation}
So we shall define 
$$
\dsl{g}{0}{i(0,n)}={\rm Id}_{X} \in G \quad (n \ge 0),
$$
and it follows that 
$$
r^0_{i(0,n)} \coloneqq u_{i(0,n)} - \dsl{g}{0}{i(0,n)}\bbl{w}{0}
\to 0 \quad \mbox{weakly}.
$$

Set $p_1 \coloneqq p[(r^0_{i(0,n)})]$.
If $p_1=0$,  then
the profile decomposition of $(u_n)$ is accomplished at this stage and the proof is complete by letting $\Lambda = 0$ and $N(n)=i(0,n)$. 
So  we now assume that 
$p_1\neq 0$.
Then by the definition of $D[(r^0_{i(0,n)})]$,  
there exist a subsequence $(i(1,n)) \preceq (i(0,n))$, 
$\bbl{w}{1} \in X \0$ and $(\dsl{g}{1}{i(1,n)}) \subset G \ (n\ge 0)$ 
such that
\begin{align*}
&\bbl{w}{1}\in D[(r^0_{i(0,n)})], \\ 
&0< \frac{1}{2} p_1 \le \|\bbl{w}{1}\|_{X} \le 
p_1, \\
&\dsl{g}{1}{i(1,n)}^{-1} r^0_{i(1,n)} \to \bbl{w}{1} 
\mbox{\quad weakly}.
\end{align*}
And obviously, one gets 
$\dsl{g}{0}{i(1,n)}^{-1} u_{i(1,n)} \to \bbl{w}{0}$ weakly in $X$.
From \cref{Lemma;T-01_431,Lemma;T-01_433}, it follows that 
\begin{equation}
\dsl{g}{0}{i(1,n)}^{-1} \dsl{g}{1}{i(1,n)} \rightharpoonup 0
\quad \mbox{as\ }  n \toinfty. 
\end{equation}
\end{step}

\begin{step}[Inductive step]
Suppose that for some $L\in\NN$, 
\begin{alignat}{2}
&\exists\, (i(l,n)) \preceq (i(l-1,n)) 
&\quad &(0 \le l \le L, \ i(-1,n) \coloneqq n \ \mbox{formally}), \label{202103230010}\\
&\exists\, \bbl{w}{l}\in X &\quad &(0 \le l \le L),  \label{202103230020}\\
&\exists\, \dsl{g}{l}{i(L,n)} \in G &\quad &(0 \le l \le L, \ n\in\NN), \label{202103230030}\\
&\exists\, r^{l}_{i(L,n)}\in X  &\quad &(0 \le l \le L, \ n\in\NN), \label{202103230040}\\
&\exists\, p_l>0 &\quad &(1 \le l \le L), \label{202103230050}
\end{alignat}
such that 
\begin{alignat}{2}
& \dsl{g}{0}{i(L,n)}={\rm Id}_{X} \  (n\ge 0), &\quad &\bbl{w}{l}\neq 0 \ (1 \le l \le  L),  \label{202103230060}\\
&\dsl{g}{l}{i(L,n)}^{-1} \dsl{g}{k}{i(L,n)} \rightharpoonup 0 
 &\quad &(n\toinfty, \ 0\le l \neq k \le L), \label{202103230070}\\
&\dsl{g}{l}{i(L,n)}^{-1} u_{i(L,n)} \to \bbl{w}{l} 
&\quad &\mbox{weakly} \  (0 \le l \le  L), \label{202103230080}\\
&r^{l}_{i(L,n)}\coloneqq u_{i(L,n)}-\sum_{k=0}^l \dsl{g}{k}{i(L,n)}\bbl{w}{k} &\quad &(0 \le l \le  L), \label{202103230090}\\
&\dsl{g}{k}{i(L,n)}^{-1} r^{l}_{i(L,n)} \to 0 
&\quad  &\mbox{weakly} \ (0 \le k \le l \le L), \label{202103230100} \\ 
&\dsl{g}{l}{i(L,n)}^{-1} r^{l-1}_{i(L,n)} \to \bbl{w}{l} 
&\quad  &\mbox{weakly} \  (1 \le l \le L),\notag \\ 
& 0<\frac{p_l}{2} \le \|\bbl{w}{l}\|_{X} \le p_l &\quad &(1 \le l \le L),  \label{202103230120}
\end{alignat}
where $p_l\coloneqq p[(r^{l-1}_{i(l-1,n)})] \ (1 \le l \le  L)$.

We may assume that 
$p_{L+1}\coloneqq p[(r^{L}_{i(L,n)})]\neq 0$; otherwise, profile decomposition is accomplished at this stage by letting $\Lambda = L$ and $N(n) \coloneqq i(L,n)$. 
Then by the definition of $p_{L+1}$,  there exist 
\begin{align}
&(i(L+1,n)) \preceq (i(L,n)), \label{202103230130} \\
&\bbl{w}{L+1}\in X\0, \notag \\
&\dsl{g}{L+1}{i(L+1,n)}\in G \ (n\ge 0), \notag \\ 
\intertext{such that} 
&\dsl{g}{L+1}{i(L+1,n)}^{-1} r^{L}_{i(L+1,n)} \to \bbl{w}{L+1}
\quad \mbox{weakly}, \label{202103230140}\\
&0<\frac{p_{L+1}}{2} \le \|\bbl{w}{L+1}\|_{X} \le p_{L+1}. \notag 
\end{align}
By \eqref{202103230080} and \eqref{202103230130}, we get 
\begin{equation}\label{202103230150}
\dsl{g}{l}{i(L+1,n)}^{-1} u_{i(L+1,n)} \to \bbl{w}{l}
\quad \mbox{weakly} \  (n\toinfty,\  0 \le l \le L). 
\end{equation}
Also, by \eqref{202103230070} and \eqref{202103230130}, we get 
\begin{equation}\label{202103230160}
\dsl{g}{k}{i(L+1,n)}^{-1}\dsl{g}{l}{i(L+1,n)} 
\rightharpoonup 0 \quad (n\toinfty, \ 0 \le l \neq k\le L).
\end{equation}
From \cref{Lemma;T-01_431,Lemma;T-01_433,Lemma;T-01_434} and from~\eqref{202103230140}--\eqref{202103230160},
we observe that 
\begin{equation}\label{202103230170}
\l\{
\begin{aligned}
\dsl{g}{l}{i(L+1,n)}^{-1} \dsl{g}{L+1}{i(L+1,n)} &\rightharpoonup 0
\quad (n\toinfty, \ 0 \le l \le L),  \\
\dsl{g}{L+1}{i(L+1,n)}^{-1} \dsl{g}{l}{i(L+1,n)} 
&=(\dsl{g}{l}{i(L+1,n)}^{-1} \dsl{g}{L+1}{i(L+1,n)} )^{-1} \\
&\rightharpoonup 0
\quad (n\toinfty, \ 0 \le l \le L).
\end{aligned}
\r. 
\end{equation}
So from~\eqref{202103230160} and~\eqref{202103230170}, it follows that 
\begin{equation}\label{202103230180}
\dsl{g}{k}{i(L+1,n)}^{-1} \dsl{g}{l}{i(L+1,n)} 
\rightharpoonup 0 
\quad (n \toinfty, \ 0\le l \neq k \le L+1).
\end{equation}
Also, from~\eqref{202103230090},~\eqref{202103230140} and~\eqref{202103230170}, it follows that 
\begin{equation}\label{202103230190}
\dsl{g}{L+1}{i(L+1,n)}^{-1}u_{i(L+1,n)} \to \bbl{w}{L+1}
\quad \mbox{weakly}.
\end{equation}

Set the $(L+1)$-st residual term  as 
\begin{equation*}
r^{L+1}_{i(L+1,n)} \coloneqq u_{i(L+1,n)}-
\sum_{l=0}^{L+1}\dsl{g}{l}{i(L+1,n)}\bbl{w}{l} \quad (n\ge 0),
\end{equation*}
and then, we observe from~\eqref{202103230100},~\eqref{202103230180} and~\eqref{202103230190} that 
\begin{equation}\notag 
\dsl{g}{l}{i(L+1,n)}^{-1} r^{L+1}_{i(L+1,n)} \to 0
\quad \mbox{weakly} \ (0 \le l \le L+1).
\end{equation}
Thus one has obtained: 
\begin{alignat*}{2}
& \exists\, (i(l,n)) \preceq (i(l-1,n))  
&\quad  &(0 \le l \le  L+1, \ i(-1,n) \coloneqq n \  \mbox{formally}), \\
&\exists\, \bbl{w}{l}\in X  &\quad  &(0 \le l \le L+1), \\
&\exists\, \dsl{g}{l}{i(L+1,n)} \in G  &\quad  &(0\le l \le L+1, \ n\in\NN), \\
&\exists\, r^{l}_{i(L+1,n)}\in X  &\quad  &(0\le l \le L+1, \ n\in\NN), \\
&\exists\, p_l>0  &\quad  &(1 \le l \le  L+1), 
\end{alignat*}
such that 
\begin{alignat*}{2}
 &\dsl{g}{0}{i(L+1,n)}={\rm Id}_{X} &\quad  &(n\ge 0),  \\ 
 &\bbl{w}{l}\neq 0 &\quad &(1 \le l \le  L+1), \\
 &\dsl{g}{l}{i(L+1,n)}^{-1} \dsl{g}{k}{i(L+1,n)} \rightharpoonup 0  &\quad &(n\toinfty, \ 0 \le l\neq k \le L+1), \\
 &\dsl{g}{l}{i(L+1,n)}^{-1} u_{i(L+1,n)} \to \bbl{w}{l}  &\quad  &\mbox{weakly} \  (0 \le l \le L+1), \\
 &r^{l}_{i(L+1,n)}\coloneqq u_{i(L+1,n)}-\sum_{k=0}^l \dsl{g}{k}{i(L+1,n)}\bbl{w}{k}  &\quad  &(0 \le l \le  L+1), \\
 &\dsl{g}{k}{i(L+1,n)}^{-1} r^{l}_{i(L+1,n)} \to 0 &\quad  &\mbox{weakly} \ (0 \le k  \le l \le  L+1), \\
 &\dsl{g}{l}{i(L+1,n)}^{-1} r^{l-1}_{i(L+1,n)} \to \bbl{w}{l}  &\quad  &\mbox{weakly} \  (1 \le l \le L+1), \\
 &0<\frac{p_l}{2} \le \|\bbl{w}{l}\|_{X} \le p_l  &\quad  &(1 \le l \le L+1).
\end{alignat*}
\end{step}

\begin{step}[Conclusion of induction]
Owing to Steps~1 and~2, one can conclude:
for any $L\in\NN^{<\Lambda +1}$ with $\Lambda$ being  defined later, \eqref{202103230010}--\eqref{202103230130} hold true.

The number of nontrivial profiles $\Lambda \in \N\cup\{0,+\infty\}$ is defined as follows. 
If the iterative procedure above ends up at the ($L+1$)-st step  for some $L\in\N$ with $p_{L+1}$ being zero for the first time,  then $\Lambda = L$, i.e., $L$ is the number of nontrivial profiles except for the principal profile $\bbl{w}{0}$ which always exists and is possibly zero. Otherwise, the number of nontrivial profiles is infinite if the iterative procedure does not end up at all.  
\end{step}

\begin{step}[Condition~\eqref{eq;Ishiwata condition X}]
We prove the following claim.

\begin{claim}
For any $0 \le L < \Lambda$ and any $\phi\in U\coloneqq B_{X^*}(1)$, 
\begin{equation}
\label{202103230200}
\varlimsup_{n\toinfty} \sup_{g \in G } \qty|\la  \phi,  g^{-1} r^L_{i(L,n)}\ra |
\le 2 \|\bbl{w}{L+1}\|_{X}.
\end{equation}
Moreover, if $\Lambda <\infty$, then for any $\phi\in U$,
\begin{equation*}
\varlimsup_{n\toinfty} \sup_{g \in G } \qty|\la  \phi,  g^{-1} r^\Lambda_{i(L,n)}\ra |
=0.
\end{equation*}
\end{claim}

\begin{proof}[Proof of Claim]
Let $\phi \in U\0$ and set 
$$
\gamma_L(\phi) \coloneqq 
\varlimsup_{n\toinfty} \sup_{g \in G } \qty|\la  \phi,  g^{-1} r^L_{i(L,n)}\ra |.
$$
If $\Lambda$ is finite and $L=\Lambda$,  then by the definition of $\Lambda$, we have  $D[(r^{\Lambda}_{i(\Lambda,n)})]=\{0\}$, which means that there are no more nontrivial profiles and that $r^{\Lambda}_{i(\Lambda,n)} \to 0$ $G$-weakly in $X$. Hence by the definition of $G$-weak convergence, we obtain $\gamma_L(\phi)=0$. 

Nextly, we suppose that $\Lambda = +\infty$ and that $0 \le L < \Lambda$.  
If $\gamma_L(\phi)=0$,  then the conclusion is clearly true. 
So we suppose $\gamma_L(\phi)>0$.
Then there exist a subsequence $(i(L,n_j))_j\preceq (i(L,n))$
and $(g_{j} )_j\subset G$ such that 
\begin{equation}
\label{202103230210}
\qty| \la  \phi, g_{j}^{-1} r^L_{i(L,n_j)} \ra | = \gamma_L(\phi)+o(1) \quad \mbox{as} \ j\toinfty.
\end{equation}
Since  all of $g_{j}$ are isometries, it follows that 
$$
\|g_{j}^{-1} r^L_{i(L,n_j)}\|_{X}
= \|r^L_{i(L,n_j)}\|_{X}
\le  \sup_{n\ge 0} \| u_{n} \|_X + \sum_{l=0}^L \|\bbl{w}{l} \|_X <\infty.
$$
So passing to a subsequence, still denoted by $(i(L,n_j))$,  
there exists a weak limit $v\in X$ such that 
$g_{j}^{-1} r^L_{i(L,n_j)} \to v$ weakly. 
It follows from~\eqref{202103230210} that 
$$
0<\gamma_L(\phi) = \lim_{j\toinfty} \qty| \la  \phi, g_{j}^{-1} r^L_{i(L,n_j)} \ra |  
= | \la  \phi, v \ra | 
\le \|v\|_{X}. 
$$
Hence one gets $v \neq 0$. 

From the definition of 
$p_{L+1}=p[(r^{L}_{i(L,n)})]>0$, follows 
$$
\|v\|_{X} \le p_{L+1} \le 2\|\bbl{w}{L+1}\|_{X}.
$$
Therefore, we get  
$$
\gamma_L(\phi) \le \|v\|_{X} \le 2\|\bbl{w}{L+1}\|_{X}.
$$
Thus the claim is proved. 
\end{proof}
\end{step}

\begin{step}[Diagonal argument]
At this stage, we shall take the diagonal sequence of 
a family of subsequences $(i(l,n)) \s (l \in \NN^{<\Lambda+1}, n \in \NN)$ 
when $\Lambda = +\infty$, say $(i(n,n))$. 
But we should pay careful attention to the well-definedness of 
corresponding dislocations  $\dsl{g}{l}{i(n,n)}\in G$.
The point is that 
the diagonal sequence except for the first $l$ terms, that is  $(i(n,n))_{n \ge l}$,  
is a subsequence of $(i(l,j))_{j\ge l}$. 
In other words, for each $l \in \NN^{<\Lambda +1}$, 
one can find an order preserving mapping 
$\Phi_l : \Z_{\ge l} \to \Z_{\ge l}$
such that 
$$
i(n,n) = i(l, \Phi_l(n)), \quad n \ge l,
$$
which means that $\Phi_l$ pulls $(i(n,n))_{n \ge l}$
back to $(i(l,j))_{j \ge l}$.

Also note that the first $l$ terms of $(i(n,n))_{n\ge 0}$ are not necessarily included into the set $\{i(l,j); \ j\ge 0\}$. 
This observation implies that 
for $l\in \NN^{<\Lambda+1}$ fixed, 
the dislocations $\dsl{g}{l}{i(n,n)} \in G$ 
are well-defined (respectively, ill-defined)
if $n \ge l$ (respectively, $n < l$), 
since $\dsl{g}{l}{i(n,n)}$ are defined originally along the indices $i(l,j)$. 
This fact causes some complicated restrictions on the indices appearing in \cref{theorem;ProDeco in X}. 
By the way, one can make $\dsl{g}{l}{i(n,n)}$ ($l \in \NN^{<\Lambda +1}$ and $n <l$) well-defined by resetting trivially 
$\dsl{g}{l}{i(n,n)} = {\rm Id}_{X}$ for $ n < l$. 
However, this does not make any drastic improvements compared to our theorem. 
\end{step}

\begin{step}[Conclusion]
Set for $n\in \NN$, 
\begin{equation*}
    N(n) \coloneqq 
    \begin{cases}
        i(\Lambda, n) &\mbox{if}  \ \Lambda <+\infty, \\
        i(n,n) &\mbox{if} \ \Lambda =+\infty.
    \end{cases}
\end{equation*}
By construction, one can immediately check that 
$\dsl{g}{0}{N(n)} = {\rm Id}_{X}$ $(n\ge 0)$. 
For each $l\in \NN^{<\Lambda+1}$ fixed,  
take  an order preserving injection  
$\Phi_l:\Z_{\ge l} \to \Z_{\ge l}$ such that 
$$
N(n)=i(l,\Phi_l (n)), \quad n\ge l.
$$
Take $l\neq k \in \NN^{<\Lambda+1}$ and 
set $\mu\coloneqq \max (l,k)$.
Then it follows from \eqref{202103230070} that 
$$
\dsl{g}{l}{N(n)}^{-1} \dsl{g}{k}{N(n)}
=
\dsl{g}{l}{i(\mu,\Phi_\mu(n))}^{-1} \dsl{g}{k}{i(\mu,\Phi_\mu(n))}
\rightharpoonup 0 \quad (n\toinfty).
$$
In the same way, for any $l\in\NN^{<\Lambda+1}$,  one obtains that 
$$
\dsl{g}{l}{N(n)}^{-1} u_{N(n)} 
=\dsl{g}{l}{i(l,\Phi_l(n))}^{-1} u_{i(l,\Phi_l(n))}
\to \bbl{w}{l} \quad \mbox{weakly}  \  (n\toinfty).
$$

Set the $L$-th residual term as 
$$
r^L_{N(n)} \coloneqq u_{N(n)}-\sum_{l=0}^L \dsl{g}{l}{N(n)} \bbl{w}{l}, \quad L \in \NN^{<\Lambda+1}, 
$$
which is well-defined if $n\ge L$.
Then~\eqref{eq;ProDeco in X 1} immediately follows from~\eqref{202103230100}.
We pass to the limit in~\eqref{202103230200} along the subsequence $(N(n))$  to get 
$$
\sup_{\phi\in U}\varlimsup_{n\toinfty}\sup_{g\in G }
\qty|\la  \phi, g^{-1}r^L_{N(n)} \ra | 
\le
\begin{cases}
2 \|\bbl{w}{L+1}\|_{X} &(0 \le L < \Lambda \le \infty), \\ 
0 &(L=\Lambda <\infty). 
\end{cases}
$$
Thus if either $\| \bbl{w}{l} \|_X \to 0$ as $l\to \Lambda = \infty$ or else $\Lambda < \infty$, then we obtain 
$$
\lim_{L\to \Lambda}\sup_{\phi\in U}\varlimsup_{n\toinfty}\sup_{g\in G}
\qty|\la  \phi, g^{-1}r^L_{N(n)} \ra | 
=0.
$$
\end{step}
From Steps~1--6, the proof is totally complete. 
\end{proof}

\section[PD in Sobolev spaces]{Profile decomposition in inhomogeneous Sobolev spaces}\label{section;ProDecoInhomSobolev}
\subsection{Settings}
Main targets in this paper are Sobolev spaces, for which theorems of profile decomposition will be  established.
Let $\Omega$ be an open set in $\RN$ $(N\in\N)$. 
For $1<p<\infty$, $N\in\N$ and $m\in\NN$, we denote by $W^{m,p}(\Omega)$ the inhomogeneous Sobolev space defined as the set of functions in  $L^p(\Omega)$ whose distributional derivatives with order up to $m$ are also in $L^p(\Omega)$. The norm of $W^{m,p}(\Omega)$ is denoted by 
$$
\|u\|_{W^{m,p}(\Omega)}=\l( \sum_{|\alpha|\le m} 
\| \d^{\alpha} u\|_{L^p(\Omega)}^p \r)^{1/p},
$$
where for a multi-index $\alpha = (\alpha_1, \ldots,\alpha_N) \in (\NN)^N$,  
$$
|\alpha|=\sum_{n=1}^N \alpha_n \mbox{\quad and\quad}
\d^\alpha = \frac{\d^{|\alpha|}}{\d x_1^{\alpha_1}\cdots \d x_N^{\alpha_N}},
$$
and derivatives are in the sense of distributions.
We denote by $\ppm$ the Sobolev critical exponent, that is, 
$\ppm=pN/(N-mp)$ if $N>mp;$ $ \ppm =\infty$ if $N\le mp$.
The reader ought to be careful in reading~\cite{T-01} 
    because Tintarev uses other notation for Sobolev spaces. 
    He denotes the Sobolev spaces defined as above by $H^{m,p}(\Omega)$.

We next provide \emph{actions of dislocations} defined on function spaces. 

\begin{definition}\label{definition;dislocation of Sobolev}
Define a group action on $L^1_{\rm loc}(\RN)$ as follows: 
\begin{align*}
&G[\ZN]
\coloneqq  
\qty{
g[y] : L^1_{\rm loc}(\RN) \to L^1_{\rm loc}(\RN) 
\left| \ 
\begin{aligned}
&g[y] u(\cdot) \coloneqq  u(\cdot - y), \\ 
&u \in L^1_{\rm loc}(\RN), \ y \in \ZN 
\end{aligned}
\right.
}. 
\end{align*}
This is a group of bijective isometries on the Sobolev space $\Wmp$  defined as above:
\begin{equation*}
\|g[y] u\|_{\Wmp} = \|u\|_{\Wmp},  \quad  u\in\Wmp, \ y\in\ZN. 
\end{equation*}
Moreover, inverse mappings 
are given as follows.
For $g[y],g[z]\in G[\ZN]$,  $u\in \Wmp$,  
\begin{align*}
g[y]^{-1}u(\cdot)&=u(\cdot+y)=g[-y]u(\cdot), \\
g[z]^{-1} g[y] u (\cdot) &= u(\cdot-(y-z))=g[y-z]u(\cdot).
\end{align*}
\end{definition}

\begin{remark}
\rm 
One can replace $\ZN$ with $\RN$ in the above dislocations and all of the following theorems. 
\end{remark}

The following lemma characterizes the operator-weak convergence of dislocations on Sobolev spaces. 

\begin{lemma}\label{lemma;Sobolev dsl weak conv}
Let $1< p<\infty$ and let  $(g[y_n])$ and $(g[z_n])$ be  
sequences in $G[\ZN]$. As bounded operators 
on $\Wmp$, the operator-weak convergence $g[y_n]\rightharpoonup 0$ is 
equivalent to $|y_n|\toinfty$,  and $g[z_n]^{-1} g[y_n] \rightharpoonup 0$
is equivalent to $|y_n-z_n|\toinfty$.
\end{lemma}

\begin{proof}
See~\cite[Lemmas~3.1  and~5.1]{T-K}.
\end{proof}

Under these settings, Sobolev spaces form dislocation spaces.

\begin{proposition}\label{prop;Wmp dislocation space}
Let $p\in ]1,\infty[$ and let  $N,m\in \N$.
Then $\qty(\Wmp,G[\ZN])$ is a dislocation space.
\end{proposition}

\begin{proof}
We shall show that $G[\ZN]$ is a dislocation group. 
It is easily checked that for any sequence $(g[y_n])$ and $g[y]$ in $G[\ZN]$, $g[y_n] \to g[y]$ operator-strongly if and only if $|y_n-y|\to 0$ as $n\toinfty$, and thus, one sees that $g[y_n] \to g[y]$ operator-strongly if and only if $g[y_n]^{-1} \to g[y]^{-1}$ operator-strongly. 
Let $(g[y_n])$ be a sequence in $G[\ZN]$ such that 
$g[y_n] \not\rightharpoonup 0$ as $n\toinfty$,  which 
equivalently means that $|y_n| \not\to \infty$ 
due to \cref{lemma;Sobolev dsl weak conv}. 
So $(y_n)$ is a bounded sequence in $\ZN$, and then 
there exist a subsequence, again denoted by $n$, and $y\in\ZN$ such that 
$y_n = y, \  n\in\N$. 
Hence one sees that $g[y_n] \to g[y]$ operator-strongly. 
Finally, we shall show that 
$(g[y_n])^* \to (g[y])^*$ operator-strongly (recall \cref{DslGrpRem}).
Let $u\in\Wmp$ and $\phi\in [\Wmp]^*$.
The dual of Sobolev spaces are completely characterized with the help of the Riesz representation theorem (see~\cite[Proposition~9.20]{B} for instance). 
Take functions $\phi_{\alpha}\in L^{p'}(\RN), \ |\alpha|\le m$,
as in~\cite[Proposition~9.20]{B}.
It follows from the change of variables and the H\"older inequality that 
\begin{align*}
&\| (g[y_n])^*\phi-(g[y])^*\phi \|_{[\Wmp]^*}  \\
&=
\sup_{ \substack{ u\in\Wmp\\ \norm{u}_{\Wmp}=1 } } \qty|\la  (g[y_n])^*\phi-(g[y])^*\phi, u\ra | \\
&=
\sup_{ \substack{ u\in\Wmp\\ \norm{u}_{\Wmp}=1 } }
\qty|\la  \phi, g[y_n]u-g[y]u \ra | \\
&=
\sup_{ \substack{ u\in\Wmp\\ \norm{u}_{\Wmp}=1 } }
\qty|\sum_{|\alpha|\le m} \int_{\RN} \phi_{\alpha} \d^{\alpha} (g[y_n]u-g[y]u)  \,\dd x  | \\
&=
\sup_{ \substack{ u\in\Wmp\\ \norm{u}_{\Wmp}=1 } }
\qty|\sum_{|\alpha|\le m} \int_{\RN} ( g[-y_n]\phi_\alpha -g[-y]\phi_\alpha ) \d^{\alpha} u  \,\dd x  | \\
&\le 
\sup_{ \substack{ u\in\Wmp\\ \norm{u}_{\Wmp}=1 } }
\qty(\sum_{|\alpha|\le m} \|g[-y_n]\phi_\alpha -g[-y]\phi_\alpha\|_{L^{p'}(\RN)} ) \| u \|_{\Wmp} \\
&=
\sum_{|\alpha|\le m} \|g[-y_n]\phi_\alpha -g[-y]\phi_\alpha\|_{L^{p'}(\RN)}. 
\end{align*}
The last term is convergent to zero as $n\toinfty$ 
since $g[-y_n]\phi_\alpha \to g[-y]\phi_\alpha$, $|\alpha|\le m$, strongly in $L^{p'}(\RN)$.
Hence one gets 
$(g[y_n])^*\phi-(g[y])^*\phi \to 0$ in $[\Wmp]^*$, 
whence follows $(g[y_n])^* \to (g[y])^* $ operator-strongly.
This completes the proof.
\end{proof}

Note that~\cref{cor;dsl sp equiv norm} ensures that Sobolev spaces with the above dislocations are dislocation spaces irrelevantly to the choice of equivalent norms that are invariant under translation. 

\medskip
We then recall that $\Wmp$ is embedded into some Lebesgue spaces $G[\ZN]$-completely continuously. 
Also, the following lemma is a reformulation of a well-known lemma named Lions' lemma~\cite[Lemma~I,1]{Lions 84 2} in the language of $G$-weak convergence.

\begin{lemma}\label{lemma;G-complete continuity of Wmp}
Let $\qty(\Wmp,G[\ZN])$ be as in \cref{prop;Wmp dislocation space}.
Then for any $q\in]p,\ppm[$,  the continuous embedding 
$\Wmp\hookrightarrow L^q(\RN)$
is $G[\ZN]$-completely continuous.
Indeed, there exists a constant $C>0$ such that 
\begin{equation*}
\int_{\RN} |u(x)|^q  \,\dd x  \le C \|u\|_{\Wmp}^p
\l( \sup_{y\in\ZN} \int_{y+Q} |u(x)|^q  \,\dd x   \r)^{1-p/q}
\end{equation*}
for all $u \in \Wmp$,  where $Q\coloneqq ]0,1[^N$.  
Moreover, 
if $u_n\to 0$ $G[\RN]$-weakly in $\Wmp$, then 
\begin{equation*}
\sup_{y\in\ZN} \int_{y+Q} |u_n(x)|^q  \,\dd x  \to 0.
\end{equation*}
\end{lemma}

\begin{proof}
See~\cite{T-01,T-K}.
\end{proof}

Generalizing the above lemma, one observes the vanishing of lower order derivatives of $u_n$ when $u_n$ goes to zero $G[\ZN]$-weakly.

\begin{lemma}\label{lemma;G-complconti-lowerorderderiv}
Let $\qty(\Wmp,G[\ZN])$ be as in \cref{prop;Wmp dislocation space}.
Then for any multi-index $\alpha\in(\NN)^N$ with $|\alpha|<m$ and
$q\in]p,p^*_{m-|\alpha|}[$, the continuous linear operator  
$\d^{\alpha}: \Wmp\hookrightarrow L^q(\RN)$
is $G[\ZN]$-completely continuous.
Indeed, there exists a constant $C>0$ such that 
\begin{equation*}
\int_{\RN} |\d^{\alpha}u(x)|^q  \,\dd x  \le C \|u\|_{\Wmp}^p
\l( \sup_{y\in\ZN} \int_{y+Q} |\d^{\alpha}u(x)|^q  \,\dd x   \r)^{1-p/q}
\end{equation*}
for all $u \in \Wmp$,  where $Q\coloneqq ]0,1[^N$.  
Moreover,
if $u_n \to0$ $G[\RN]$-weakly in $\Wmp$, then 
\begin{equation*}
\sup_{y\in\ZN} \int_{y+Q} |\d^{\alpha}u_n(x)|^q  \,\dd x  \to 0.
\end{equation*}
\end{lemma}

\begin{proof}
The preceding lemma and the Sobolev compact embedding 
$\d^{\alpha}: W^{m,p}(Q)\to L^p(Q)$
yield the conclusion.
\end{proof}

A further generalization of the above lemmas reads: 

\begin{lemma}\label{lemma;G-complconti-lowerorderderiv2}
Let $\qty(\Wmp,G[\ZN])$ be as in \cref{prop;Wmp dislocation space} and let $k\in\NN$ be such that $k<m$. 
Then for all $q\in ]p,p^*_{m-k}[$, the continuous embedding 
$W^{m,p}(\RN)\hookrightarrow W^{k,q}(\RN)$ is $G[\ZN]$-completely continuous.
\end{lemma}

\begin{proof}
Use the preceding lemmas and the increasing property of the Sobolev critical exponent: $p^*_{m_1}\le p^*_{m_2}$ if $m_1 \le m_2$.
\end{proof}

\subsection{Main theorem}
We are in a position to state a theorem of profile decomposition in the \emph{inhomogeneous} Sobolev space $\Wmp$. The following theorem indicates that every bounded sequence in $\Wmp$ has a `fine' profile decomposition in the sense that the residual term is vanishing in $L^q(\RN), \ q \in ]p, \ppm[$. Also, one can obtain energy decompositions in Sobolev and Lebesgue norms.

\begin{theorem}\label{theorem;profile-decomp. in Wmp}
Let $1 < p< \infty$, let $m,N\in\N$ and let $(u_n)$ be a bounded sequence in $\Wmp$.
Then there exist a number $\Lambda \in \N\cup \{ 0,+\infty\}$, 
a subsequence $(N(n)) \subset \NN \ (n \in \NN)$,  
profiles 
$\bbl{w}{l} \in \Wmp \ (l \in \NN^{<\Lambda +1})$,  
vectors
$\dsl{y}{l}{N(n)} \in \ZN \ (l \in \NN^{<\Lambda +1}, \  n \in \Z_{\ge l})$, 
and residual terms $r^L_{N(n)}\in\Wmp$ $(L\in\NN^{<\Lambda+1}, \ n\in\Z_{\ge L})$ with the relation of a double-suffix profile decomposition 
\begin{align*}
u_{N(n)}=\sum_{l=0}^L \bbl{w}{l}(\cdot-\dsl{y}{l}{N(n)})+r^L_{N(n)}, \quad L\in\NN^{<\Lambda+1}, \ n\in\Z_{\ge L}, 
\end{align*}
such that the following hold.
\begin{enumerate}
\setlength{\itemsep}{0mm}
\setlength{\parskip}{0mm}

    \item $\dsl{y}{0}{N(n)}=0 \ (n \ge 0), \quad \bbl{w}{l}\neq 0 \ (1 \le l \in \NN^{<\Lambda +1}).$
    \item $|\dsl{y}{l}{N(n)}-\dsl{y}{k}{N(n)}|\toinfty$
    as $n \toinfty$
    whenever $l\neq k \in \NN^{<\Lambda +1}.$
    \item $u_{N(n)}(\cdot+\dsl{y}{l}{N(n)}) \to \bbl{w}{l}$ 
    weakly in $\Wmp$ \   
    \mbox{and a.e.\  on} $\RN \ (n\toinfty, \ l \in \NN^{<\Lambda +1}).$ 
    \item 
    For $k\in\NN^{<\Lambda+1}$, 
    \begin{equation*}
    r^{L}_{N(n)} (\cdot + \dsl{y}{k}{N(n)}) \to 
    \begin{cases}
    0  &  \mbox{if}  \  k=0,\ldots, L, \\
    \bbl{w}{k}  &  \mbox{if} \ k \ge L+1, 
    \end{cases}
    \end{equation*}
    weakly in $\Wmp$ and a.e.\ on $\RN$ as $n\toinfty$.  
\end{enumerate} 
Moreover, the following hold: 
\begin{align}
&\label{eq;energy estim Wmp}
\varlimsup_{n\toinfty}\|u_{N(n)} \|_{\Wmp}^p \\ 
&\notag 
\ge \sum_{l=0}^\Lambda \|\bbl{w}{l}\|_{\Wmp}^p
+ \varlimsup_{L\to \Lambda} \varlimsup_{R\toinfty} \varlimsup_{n\toinfty} 
\| r^L_{N(n)} \|_{W^{m,p}(\RN \setminus \Anrl ) },
\\
&\lim_{L\to \Lambda}\sup_{\phi\in U}\varlimsup_{n\toinfty}\sup_{y\in\ZN}
\qty|\la \phi, r^L_{N(n)}(\cdot+y)\ra_{\Wmp} |=0, \label{eq;Ishiwata condition Wmp}\\ 
&\varlimsup_{L\to \Lambda}\varlimsup_{n\toinfty}\|r^{L}_{N(n)}\|_{\Wmp} \le 2 \varlimsup_{n\toinfty}\|u_{N(n)} \|_{\Wmp} < +\infty, \label{eq;residue bounded Wmp}\\
&\lim_{L\to \Lambda}\varlimsup_{n \toinfty} \|\d^\alpha r^L_{N(n)} \|_{L^q(\RN)}=0, \quad  q\in ]p,p^*_{m-|\alpha|}[, \ \alpha\in(\NN)^N, \ |\alpha|<m, \label{eq;residue vanish Wmp} \\
&\lim_{L\to \Lambda}\varlimsup_{n \toinfty} \|r^L_{N(n)} \|_{W^{k,q}(\RN)}=0, \quad  q\in ]p,p^*_{m-k}[, \ k\in\NN, \ k<m, \label{eq;residue vanish Wmp2}, \\ 
&\varlimsup_{n\toinfty} \|u_{N(n)}\|_{L^q(\RN)}^q 
=\sum_{l=0}^\Lambda \|\bbl{w}{l}\|_{L^q(\RN)}^q 
+\lim_{L\to \Lambda}\varlimsup_{n\toinfty}\|r^L_{N(n)}\|_{L^q(\RN)}^q, \quad   q\in [p,\ppm], \label{eq;Brezis Lieb theorem3}
\end{align}
where 
$U \coloneqq B_{[W^{m,p}(\RN)]^*}(1)$ and  
$\Anrl \coloneqq \bigcup_{l=0}^L B( \dsl{y}{l}{N(n)}, R).$
\end{theorem}

We see important remarks: 

\begin{remark} 
\begin{enumerate}
\item 
The meaning of each assertion above is as follows:  
\eqref{eq;energy estim Wmp} indicates that the sum of ($p$-powered) Sobolev norms of all profiles is bounded by  
$\varlimsup_{n\toinfty} \|u_{N(n)} \|_{\Wmp}$, 
which is, so to speak, an \emph{energy estimate} or an \emph{energy decomposition}; \eqref{eq;Ishiwata condition Wmp} implies the completion of performing the profile decomposition and called the exactness condition; \eqref{eq;residue bounded Wmp} ensures that the residual term does not diverge as the number of subtracted dislocated profiles increases; \eqref{eq;residue vanish Wmp} and~\eqref{eq;residue vanish Wmp2} show that lower order derivatives of the residual term are vanishing strongly in appropriate Lebesgue or Sobolev spaces; \eqref{eq;Brezis Lieb theorem3} implies that a basic form of the iterated Brezis-Lieb lemmma holds true, possibly with limiting residual term of positive mass in the case $q=p,\ppm$.

\item When one considers profile decomposition in Sobolev spaces $W^{m,2}(\RN)$, which is a Hilbert space, one should employ \cref{theorem;general profile decomp.} rather than the above theorem, since it provides a more precise energy decomposition for Sobolev norms, i.e., one obtains 
\begin{equation*}
    \varlimsup_{n\toinfty} \|u_{N(n)}\|_{W^{m,2}(\RN)}^2
    = \sum_{l =0}^\Lambda \|\bbl{w}{l}\|_{W^{m,2}(\RN)}^2+ 
    \lim_{L\to \Lambda} \varlimsup_{n\toinfty} \|r^{L}_{N(n)}\|_{W^{m,2}(\RN)}^2. 
\end{equation*}

\item 
Regarding the energy estimate~\eqref{eq;energy estim Wmp}, it is noteworthy that this is a sharper version than other authors' energy decompositions because other types of one do not include the ``residual part'' as in~\eqref{eq;energy estim Wmp}.

\item 
Relations~\eqref{eq;Ishiwata condition Wmp}--\eqref{eq;Brezis Lieb theorem3} are also obtained in the previous researches, e.g., \cite{Gerard,Jaffard,Solimini,T-01,T-K}. 

\end{enumerate}
\end{remark}

In view of applications, to e.g., studies of PDEs or Calculus of Variations, the number of nontrivial profiles $\Lambda$ is often finite and \emph{a priori} lower bounds of norms of profiles are also gained as a consequence of variational structures.  The global compactness results established by Struwe \cite{Struwe} are typical and important examples of the case of finite number profiles. Once $\Lambda$ turns out to be finite in \cref{theorem;profile-decomp. in Wmp}, some of assertions in the theorem become clearer and one can readily show the following corollary. 

\begin{corollary}
Suppose that the same conditions  as in \cref{theorem;profile-decomp. in Wmp} are satisfied. 
In addition, assume that $\Lambda < +\infty$, i.e., 
the number of nontrivial profiles is finite. 
Then regarding the final residual term given by 
\begin{equation*}
r^\Lambda_{N(n)} \coloneqq u_{N(n)} -\sum_{l=0}^\Lambda \bbl{w}{l}(\cdot - \dsl{y}{l}{N(n)}), 
    \quad n \ge \Lambda, 
\end{equation*}
the relations~\eqref{eq;energy estim Wmp}--\eqref{eq;Brezis Lieb theorem3} turn to 
\begin{align}
&\varlimsup_{n\toinfty}\|u_{N(n)} \|_{\Wmp}^p
    \ge \sum_{l=0}^\Lambda \|\bbl{w}{l}\|_{\Wmp}^p
    + \varlimsup_{R\toinfty} \varlimsup_{n\toinfty} 
    \| r^{\Lambda}_{N(n)} \|_{W^{m,p}(\RN \setminus \mathcal{A}_{n,R,\Lambda} ) },\notag\\
&\lim_{n\toinfty} \sup_{y\in \ZN} 
\qty|\la  \phi, r^\Lambda_{N(n)}(\cdot+y) \ra_{\Wmp} |=0, \quad \phi \in U, \label{residue G weak conv Wmp}\\
&\varlimsup_{n\toinfty}\|r^\Lambda_{N(n)}\|_{\Wmp}
\le 2 \varlimsup_{n\toinfty} \|u_{N(n)}\|_{\Wmp}, \notag \\
&\begin{alignedat}{2}
&\lim_{n\toinfty}\|\d^\alpha r^\Lambda_{N(n)}\|_{L^q(\RN)}=0, 
&\quad &q\in ]p,p^*_{m-|\alpha|}[, \ \alpha\in(\NN)^N, \ |\alpha|<m, \notag \\ 
&\lim_{n \toinfty} \|r^\Lambda_{N(n)} \|_{W^{k,q}(\RN)}=0, &\quad  &q\in ]p,p^*_{m-k}[, \ k\in\NN, \ k<m, \notag \\
\end{alignedat}\\
&\varlimsup_{n\toinfty}\|u_{N(n)}\|_{L^q(\RN)}^q 
=\sum_{l=0}^\Lambda \|\bbl{w}{l}\|_{L^q(\RN)}^q
+\varlimsup_{n\toinfty}\|r^\Lambda_{N(n)}\|_{L^q(\RN)}^q, \quad  q\in [p,\ppm],\notag 
\end{align}
where $U \coloneqq B_{[W^{m,p}(\RN)]^*}(1)$ and 
$\mathcal{A}_{n,R,\Lambda}\coloneqq  \bigcup_{l=0}^\Lambda B( \dsl{y}{l}{N(n)}, R).$ 
In particular, the relation~\eqref{residue G weak conv Wmp} implies that the final residual terms are $G[\ZN]$-weakly convergent to zero in $\Wmp$ (cf. \cref{definition;G-weak convergence}). 
\end{corollary}

In applying the profile decomposition theorem to studies of PDEs especially involving the $p$-Laplacian, the Sobolev space $W^{1,p}(\RN)$, $1 < p <\infty$, is often equipped with another norm and a modified theorem corresponding to the new norm is required. 

\begin{corollary}\label{corollary;another inhomogeneous norm}
Under the same conditions as in \cref{theorem;profile-decomp. in Wmp} on $W^{1,p}(\RN)$,  all assertions as in \cref{theorem;profile-decomp. in Wmp} hold true even when the norm of $W^{1,p}(\RN)$ is replaced by 
$$
\|u\|_{W^{1,p}(\RN)}\coloneqq \l( \int_{\RN} |u|^p  \,\dd x  
+ \int_{\RN} |\nabla u|^p  \,\dd x  \r)^{1/p}\!,  
\quad u\in W^{1,p}(\RN),
$$
where $|\nabla u|=\sqrt{|\frac{\d u}{\d x_1}|^2 + \cdots + |\frac{\d u}{\d x_N}|^2}$.
\end{corollary}

\begin{proof}
The corollary can be proved in the same way as in the proof of \cref{theorem;profile-decomp. in Wmp} which will be given below. 
\end{proof}

\subsection{Proof of \cref{theorem;profile-decomp. in Wmp}}
The proof of \cref{theorem;profile-decomp. in Wmp} will proceed mainly in three steps: (i) finding profiles; (ii) an energy decomposition for the Sobolev norm and the exactness condition of profile decomposition; (iii) vanishing of the residual term.
Throughout this proof, when a sequence $(u_n)$ converges to $u$ weakly in $\Wmp$ and almost everywhere on $\RN$, we denote it by 
\begin{equation*}
u_n \to u \quad \mbox{weakly and a.e.}
\end{equation*}
for short. 

\setcounter{step}{0}
\begin{step}[Finding profiles]
\cref{theorem;ProDeco in X} leads us to the existence of profile elements $(\bbl{w}{l}, \dsl{y}{l}{N(n)},\Lambda) \in \Wmp \times G[\ZN]\times (\N\cup\{0,\infty\})$ $( l \in \NN^{<\Lambda+1}, \ n \in \Z_{\ge l})$. In the proof of \cref{theorem;ProDeco in X}, each profile $\bbl{w}{l}$ is obtained as the weak limit of $u_{i(l,n)}(\cdot + \dsl{y}{l}{i(l,n)})$, where $(i(l,n))$ denotes the $l$-th subsequence for which the $l$-th profile and the $l$-th dislocations are considered. However,  with the help of the Sobolev compact embeddings, the weak convergence $u_n \to u$ in $\Wmp$ also implies the pointwise convergence $u_n \to u$ a.e.\  on $\RN$ up to a subsequence.
Hence by extracting additional subsequences in each step of the proof,  one can obtain both the weak convergence and the pointwise convergence for profiles and residual terms. 
Therefore, according to \cref{theorem;ProDeco in X} we have the following:   
\begin{alignat}{2}
& \dsl{y}{0}{N(n)}=0 &\quad  &(n\ge 0), \notag  \\
&\bbl{w}{l}\neq 0 &\quad &(1 \le l \in \NN^{<\Lambda +1}), \notag  \\
& \l| \dsl{y}{l}{N(n)} - \dsl{y}{k}{N(n)} \r| \toinfty &\quad &(n \toinfty, \ 0 \le k \neq l\in \NN^{<\Lambda +1}), \label{eq;100506}\\
& u_{N(n)}(\cdot+\dsl{y}{l}{N(n)}) \to \bbl{w}{l} 
&\quad &\mbox{weakly and a.e.} \  (0 \le l \in \NN^{<\Lambda +1}).  \label{202103150010}
\end{alignat}

Set the residual term 
$ r^{L}_{N(n)}=u_{N(n)}-\sum_{l=0}^L \bbl{w}{l}(\cdot - \dsl{y}{l}{N(n)})$ for $0 \le L \in \NN^{<\Lambda +1}$ and  $n \in \Z_{\ge L}$. 
Then the triangle  inequality yields 
\begin{equation}
 \sup_{n\ge 0} \|r^L_{N(n)} \|_{\Wmp}
\le 
\sup_{n\ge 0} \|u_{N(n)} \|_{\Wmp} 
+ \sum_{l=0}^L \|\bbl{w}{l}\|_{\Wmp}. 
\label{202103140010}  
\end{equation}
Also, \cref{theorem;ProDeco in X} implies 
\begin{alignat}{2}
& r^{L}_{N(n)}(\cdot+ \dsl{y}{l}{N(n)} ) \to 0
&\quad &\mbox{weakly and a.e.} \ (0 \le l \le L\in \NN^{<\Lambda +1}), \notag  \\ 
& r^{L-1}_{N(n)}(\cdot+\dsl{y}{L}{N(n)}) \to \bbl{w}{L} 
&\quad &\mbox{weakly and a.e.}  \  (1  \le L \in \NN^{<\Lambda+1}). \label{202011180020}
\end{alignat}
\end{step}

The remaining assertions~\eqref{eq;energy estim Wmp}--\eqref{eq;Brezis Lieb theorem3} 
will be proved in the following four lemmas. 

\begin{step}[Energy decomposition and exactness condition]
\begin{lemma}[Estimates~\eqref{eq;energy estim Wmp} and the exactness condition~\eqref{eq;Ishiwata condition Wmp}]
It holds that 
\begin{align}
\label{eq;200197}
&\varlimsup_{n\toinfty}\|u_{N(n)} \|_{\Wmp}^p
\\
&\notag 
\ge \sum_{l=0}^\Lambda \|\bbl{w}{l}\|_{\Wmp}^p
+ \varlimsup_{L\to \Lambda} \varlimsup_{R\toinfty} \varlimsup_{n\toinfty} 
\| r^L_{N(n)} \|^p_{W^{m,p}(\RN \setminus \Anrl  )},
\end{align}
where $\Anrl \coloneqq \bigcup_{l=0}^L B( \dsl{y}{l}{N(n)}, R ). $
Furthermore, it follows that 
\begin{equation} \label{202104300010}
 \lim_{L\to \Lambda}\sup_{\phi\in U}\varlimsup_{n\toinfty}\sup_{y\in \ZN}
\qty|\la  \phi, r^L_{N(n)}(\cdot + y) \ra |=0,
\end{equation}
where $U \coloneqq B_{[\Wmp]^*}(1)$. 
\end{lemma}

\begin{proof}
Once the energy estimate~\eqref{eq;200197} is obtained, then the exactness condition~\eqref{202104300010} will readily follow owing to \cref{theorem;ProDeco in X}.
Hence we shall prove the energy estimate. 
Fix a multi-index $\alpha \in (\NN)^N$ with $ |\alpha|\le m$.  
The mutual orthogonality condition~\eqref{eq;100506} implies that, for any $L \in \NN^{<\Lambda +1}$, 
\begin{equation*}
\Anrl = \bigsqcup_{l=0}^L B(\dsl{y}{l}{N(n)}, R) 
\quad (\mbox{a disjoint union})
\end{equation*}
for sufficiently large $n \in \N$.
It then follows that  for sufficiently large $n$,  
\begin{align}\label{202104300020}
&\int_{\RN} |\d^\alpha u_{N(n)} |^p  \,\dd x  \\  
&\notag  = \int_{\Anrl}|\d^\alpha u_{N(n)} |^p  \,\dd x   
+ \int_{\RN \setminus \Anrl}|\d^\alpha u_{N(n)} |^p  \,\dd x  \\ 
&\notag  =\sum_{l=0}^L \int_{B(\dsl{y}{l}{N(n)}, R  )  } |\d^\alpha u_{N(n)} |^p  \,\dd x 
+ \int_{\RN \setminus \Anrl}|\d^\alpha u_{N(n)} |^p  \,\dd x   \\ 
&\notag  =\sum_{l=0}^L \int_{B(0, R  )  } |\d^\alpha u_{N(n)} (\cdot + \dsl{y}{l}{N(n)} ) |^p  \,\dd x  
+ \int_{\RN \setminus \Anrl}|\d^\alpha u_{N(n)} |^p  \,\dd x  
\\ 
&\notag  \eqqcolon \sum_{l=0}^L \bbl{I}{l} + II. 
\end{align}
Owing to the convexity of $|\cdot|^p$  and~\eqref{202103150010}, one has  
\begin{align}
\label{202104300030}
\bbl{I}{l} &= 
\int_{B(0, R  )  } |\d^\alpha u_{N(n)} (\cdot + \dsl{y}{l}{N(n)} ) |^p  \,\dd x  \\
&\notag  \ge 
\int_{B(0, R  )  } |\d^\alpha \bbl{w}{l} |^p  \,\dd x 
\\
&\notag \qquad 
+ p  
\int_{B(0, R  )  } |\d^\alpha \bbl{w}{l} |^{p-2} \d^\alpha \bbl{w}{l} \d^\alpha [ u_{N(n)} (\cdot + \dsl{y}{l}{N(n)} )- \bbl{w}{l}  ]   \,\dd x  \\ 
&\notag  =
\int_{B(0, R  )  } |\d^\alpha \bbl{w}{l} |^p  \,\dd x 
+o(1)
\end{align}
as $n\toinfty$. 

On the other hand, due to the convexity again, one has 
\begin{align} \label{202105210010}
II &= 
\int_{\RN \setminus \Anrl}|\d^\alpha u_{N(n)} |^p  \,\dd x  \\ 
&\notag \ge 
\int_{\RN \setminus \Anrl} |\d^\alpha r^L_{N(n)} |^p  \,\dd x  \\  
&\notag \qquad + p \int_{\RN \setminus \Anrl} |\d^\alpha r^L_{N(n)} |^{p-2}
\d^\alpha r^L_{N(n)}  \d^\alpha \l(\sum_{l=0}^L \bbl{w}{l}(\cdot - \dsl{y}{l}{N(n)}) \r)   \,\dd x   \\ 
&\notag = 
\int_{\RN \setminus \Anrl} |\d^\alpha r^L_{N(n)} |^p  \,\dd x  \\  
&\notag \qquad + p \sum_{l=0}^L \int_{\RN \setminus \Anrl} |\d^\alpha r^L_{N(n)} |^{p-2}
\d^\alpha r^L_{N(n)}  \d^\alpha \l( \bbl{w}{l}(\cdot - \dsl{y}{l}{N(n)}) \r)   \,\dd x   \\ 
&\notag \eqqcolon 
\int_{\RN \setminus \Anrl} |\d^\alpha r^L_{N(n)} |^p  \,\dd x   
+ p \sum_{l=0}^L  \bbl{III}{l}.  
\end{align}
From the H\"older inequality and~\eqref{202103140010}, 
it follows that 
\begin{align}\label{202105210020}
|\bbl{III}{l}| 
& \le \| r^L_{N(n)} \|_{\Wmp}^{p-1} \l( 
\int_{\RN \setminus \Anrl} |\d^\alpha \bbl{w}{l}(\cdot - \dsl{y}{l}{N(n)}) |^p  \,\dd x  \r)^{1/p} \\ 
&\notag \le  C  \l( 
\int_{\RN \setminus B(\dsl{y}{l}{N(n)}, R )} |\d^\alpha \bbl{w}{l}(\cdot - \dsl{y}{l}{N(n)}) |^p  \,\dd x  \r)^{1/p} \\ 
&\notag = C  \l( 
\int_{\RN \setminus B(0, R )} |\d^\alpha \bbl{w}{l}|^p  \,\dd x  \r)^{1/p}  
\to 0, \quad \mbox{as} \ R \to \infty. 
\end{align}

Combining~\eqref{202104300020}--\eqref{202105210020}, we get 
\begin{align}
\label{202105090020}
&\int_{\RN} |\d^\alpha u_{N(n)} |^p  \,\dd x  \\  
&\notag  \ge 
\sum_{l=0}^L 
\int_{B(0, R  )  } |\d^\alpha \bbl{w}{l} |^p  \,\dd x  
 +
\int_{\RN \setminus \Anrl}|\d^\alpha r^L_{N(n)}  |^p   \,\dd x  \\ 
&\notag  \qquad  -C \sum_{l=0}^L
\l( \int_{\RN \setminus B(0,R )  } |  \d^\alpha  \bbl{w}{l} |^p  \,\dd x  \r)^{1/p}
+ o(1)
\end{align}
as $n\toinfty$. 
Summing up~\eqref{202105090020} over all multi-indices $\alpha$ with $|\alpha|\le m$ 
and passing to the limits as $n \toinfty$, $R \toinfty$ and then  $L \to \Lambda$, 
one obtains 
\begin{align*}
\varlimsup_{n\toinfty} 
\| u_{N(n)} \|_{\Wmp}^p  
\ge 
\sum_{l=0}^\Lambda  \| \bbl{w}{l} \|_{\Wmp}^p 
+
 \varlimsup_{L\to \Lambda} \varlimsup_{R \toinfty} \varlimsup_{n\toinfty} 
 \| r^L_{N(n)} \|_{ W^{m,p}(\RN \setminus \Anrl) }^p,
\end{align*}
which completes the proof. 
\end{proof}

\begin{lemma}[Boundedness of the residual terms~\eqref{eq;residue bounded Wmp}]
\label{lemma;bdd residue Wmp}

It holds that 
\begin{equation*}
    \varlimsup_{L\to \Lambda}\varlimsup_{n\toinfty} \|r^L_{N(n)}\|_{\Wmp}
    \le 2 \varlimsup_{n\toinfty}\| u_{N(n)}  \|_{\Wmp} < +\infty.
\end{equation*}
\end{lemma}

\begin{proof}
Set $S^L_{N(n)} \coloneqq \sum_{l=0}^L \bbl{w}{l} (\cdot - \dsl{y}{l}{N(n)}).$
We shall employ the following elementary inequality for the Euclidean norm:
for all $\alpha_l\in\R^d$ $(l=1,\ldots,L, \ d,L\in\N)$, 
\begin{equation}\label{202108060010}
\l| \l| \sum_{l=1}^L \alpha_l \r|^p - \sum_{l=1}^L |\alpha_l|^p \r|\le C_L \sum_{l\neq m}|\alpha_l||\alpha_m|^{p-1}
\end{equation}
for some constant $C_L>0$. This inequality will be proved as in \cref{lemma;power estimates of F} with $\eps=0$. 
From this, we see that for any multi-index $\alpha$ with $|\alpha|\le m$, 
\begin{align*}
&\l|  \int_{\RN} |\d^\alpha S^L_{N(n)}|^p  \,\dd x  
- \sum_{l=0}^L \int_{\RN} |\d^\alpha \bbl{w}{l}|^p  \,\dd x  \r| \\ 
&\le  
\int_{\RN} \l| \l| \sum_{l=0}^L \d^\alpha \bbl{w}{l}(\cdot - \dsl{y}{l}{N(n)}) \r|^p - \sum_{l=0}^L |\d^\alpha \bbl{w}{l} (\cdot - \dsl{y}{l}{N(n)})|^p \r|  \,\dd x  \\
&\le C_L  \sum_{l \neq l'} 
\int_{\RN} |\d^\alpha \bbl{w}{l}(\cdot - \dsl{y}{l}{N(n)})| |\d^\alpha \bbl{w}{l'}(\cdot - \dsl{y}{l'}{N(n)})|^{p-1}  \,\dd x .
\end{align*}
The mutual orthogonality condition~\eqref{eq;100506} implies that for any $0 \le l \neq l' \le L$, 
\begin{align}\label{202105020100}
\int_{\RN} |\d^\alpha \bbl{w}{l}(\cdot - \dsl{y}{l}{N(n)})| |\d^\alpha \bbl{w}{l'}(\cdot - \dsl{y}{l'}{N(n)})|^{p-1}  \,\dd x 
\to 0
\end{align}
as $n\toinfty$. 
Hence we get 
\begin{align*}
\lim_{n\toinfty}
\int_{\RN} |\d^\alpha S^L_{N(n)}|^p  \,\dd x  
= \sum_{l=0}^L \int_{\RN} |\d^\alpha \bbl{w}{l}|^p  \,\dd x .
\end{align*}

Adding up the above over all multi-indices $\alpha$ with $|\alpha|\le m$, we obtain 
\begin{equation*}
\| S^L_{N(n)} \|_{\Wmp}^p
= \sum_{l=0}^L \| \bbl{w}{l} \|_{\Wmp}^p +o(1) \quad (n \toinfty),
\end{equation*}
which together with~\eqref{eq;200197} implies  
\begin{equation}\label{202105120020}
\varlimsup_{L \to \Lambda} \varlimsup_{n\toinfty} 
\| S^L_{N(n)} \|_{\Wmp}^p
= \sum_{l=0}^\Lambda \| \bbl{w}{l} \|_{\Wmp}^p
\le \varlimsup_{n \toinfty} \| u_{N(n)} \|_{\Wmp}^p. 
\end{equation}
So we observe from~\eqref{202105120020} that  
\begin{align*}
&\varlimsup_{L \to \Lambda} \varlimsup_{n\toinfty} 
\|r^L_{N(n)}\|_{\Wmp} \\
&\le 
\varlimsup_{n \toinfty} \| u_{N(n)} \|_{\Wmp}
+ 
\varlimsup_{L \to \Lambda} \varlimsup_{n\toinfty} 
\| S^L_{N(n)} \|_{\Wmp} \\ 
&\le 
2 \varlimsup_{n \toinfty} \| u_{N(n)} \|_{\Wmp}, 
\end{align*}
which implies the boundedness of the double sequence 
$(r^L_{N(n)})$ in $\Wmp$.
This completes the proof. 
\end{proof}
\end{step}

\begin{step}[Vanishing of the residual term]
\begin{lemma}[Vanishing of the residual term~\eqref{eq;residue vanish Wmp} and~\eqref{eq;residue vanish Wmp2}] \label{lemma;proof residue vanish Wmp}
For any multi-index $\alpha\in(\NN)^N$ with $|\alpha|<m$ and any $q\in]p,p^*_{m-|\alpha|}[$,  
\begin{equation}
\label{10170010}
\lim_{L\to \Lambda}\varlimsup_{n\toinfty}\int_{\RN}|\d^\alpha r^L_{N(n)}|^q  \,\dd x  =0.
\end{equation}
In fact, for any $k\in\NN$ with $k<m$ and any $q\in ]p,p^*_{m-k}[$, 
\begin{equation}\label{10170011}
\lim_{L\to \Lambda}\varlimsup_{n\toinfty}
\|r^L_{N(n)}\|_{W^{k,q}(\RN)}=0.
\end{equation} 
\end{lemma}

\begin{proof}
This lemma is readily checked by \cref{theorem;weak G-comp conti} and \cref{lemma;G-complconti-lowerorderderiv,lemma;G-complconti-lowerorderderiv2} together with~\eqref{202104300010}.
\end{proof}
\end{step}

One can show a simple version of the Brezis-Lieb lemma. 
Although the following lemma will be shown by using the classical Brezis-Lieb lemma iteratively, we will give another proof.

\begin{lemma}[Simple version of the Brezis-Lieb lemma~\eqref{eq;Brezis Lieb theorem3}]
For any $q\in[p,\ppm]$,  
\begin{equation}\notag 
\varlimsup_{n\toinfty}\|u_{N(n)}\|_{L^q(\RN)}^q 
=  
\sum_{l=0}^\Lambda \|\bbl{w}{l}\|_{L^q(\RN)}^q
+\lim_{L\to \Lambda} \varlimsup_{n\toinfty}\|r^L_{N(n)}\|_{L^q(\RN)}^q.
\end{equation}
\end{lemma}

\begin{proof}
From~\eqref{202108060010}, one sees that 
\begin{align*}
&
\int_{\RN}\l| |u_{N(n)}|^q -\sum_{l=0}^L|\bbl{w}{l}(\cdot-\dsl{y}{l}{N(n)})|^q -|r^L_{N(n)}|^q  \r| \,\dd x  \\
&\le 
C_L\sum_{0\le l\neq k\le L}\int_{\RN} |\bbl{w}{l}(\cdot-\dsl{y}{l}{N(n)})||\bbl{w}{k}(\cdot-\dsl{y}{k}{N(n)})|^{q-1} \,\dd x \\
&\qquad 
+C_L\sum_{l=0}^L\int_{\RN} |\bbl{w}{l}(\cdot-\dsl{y}{l}{N(n)})||r^L_{N(n)}|^{q-1} \,\dd x 
\\
&\qquad 
+C_L\sum_{l=0}^L\int_{\RN} |\bbl{w}{l}(\cdot-\dsl{y}{l}{N(n)})|^{q-1}|r^L_{N(n)}| \,\dd x .
\end{align*}
By the mutual orthogonality condition~\eqref{eq;100506}, 
we have 
\begin{equation*}
    \int_{\RN} |\bbl{w}{l}(\cdot-\dsl{y}{l}{N(n)})||\bbl{w}{k}(\cdot-\dsl{y}{k}{N(n)})|^{q-1} \,\dd x  \to 0 \quad (n\toinfty, \ l\neq k).
\end{equation*}

Since the sequence $(r^L_{N(n)}(\cdot+\dsl{y}{l}{N(n)}))$ is convergent to zero a.e.\  in $\RN$ and is bounded in $L^q(\RN)$, one gets 
\begin{alignat*}{2}
    &|r^L_{N(n)}|\to 0 &\quad &\mbox{weakly in}\ L^q(\RN), \\ 
    &|r^L_{N(n)}|^{q-1}\to 0 &\quad &\mbox{weakly in}\ L^{q'}(\RN),
\end{alignat*}
(see, e.g.,~\cite[Lemma~1.3.1]{T-01})
and thus we get  
\begin{align*}
    &\int_{\RN} |\bbl{w}{l}(\cdot-\dsl{y}{l}{N(n)})|^{q-1}|r^L_{N(n)}| \,\dd x 
    =\int_{\RN} |\bbl{w}{l}|^{q-1}|r^L_{N(n)}(\cdot+\dsl{y}{l}{N(n)})| \,\dd x 
    \to 0, \\ 
    &\int_{\RN} |\bbl{w}{l}(\cdot-\dsl{y}{l}{N(n)})||r^L_{N(n)}|^{q-1} \,\dd x 
    =\int_{\RN} |\bbl{w}{l}||r^L_{N(n)}(\cdot+\dsl{y}{l}{N(n)})|^{q-1} \,\dd x 
    \to 0. \\ 
\end{align*}
Hence we have: 
\begin{equation*}
    \int_{\RN} |u_{N(n)}|^q \,\dd x 
    =\sum_{l=0}^L\int_{\RN}|\bbl{w}{l}|^q \,\dd x 
    +\int_{\RN}|r^L_{N(n)}|^q  \,\dd x  +o(1) \quad (n\toinfty).
\end{equation*}
Passing to the limits as $n\toinfty$ and then $L\to\Lambda$, one obtains
\begin{equation*}
    \varlimsup_{n\toinfty}\int_{\RN}|u_{N(n)}|^q  \,\dd x  
    =\sum_{l=0}^\Lambda \int_{\RN}|\bbl{w}{l}|^q  \,\dd x  
    +\lim_{L\to\Lambda}\varlimsup_{n\toinfty}\int_{\RN}|r^L_{N(n)}|^q \,\dd x .
\end{equation*}
\end{proof}

Eventually, the proof of \cref{theorem;profile-decomp. in Wmp} has been  complete.
\qed 

\section[Decomposition of integral functionals]{Decomposition of integral functionals in subcritical cases}\label{section;BLcubcrit}
Now we move on to the results of \emph{decomposition of integral functionals} subordinated to profile decomposition (also known as the iterative Brezis-Lieb lemmas) which originally state that along the profile decomposition, $\|u_n\|_{L^q(\RN)}^q$ will be decomposed into the (possibly infinite) sum of those of profiles for suitable $q$, while the classical Brezis-Lieb lemma implies that $\|u_n\|_{L^q(\RN)}^q-  \|\bbl{w}{0}\|_{L^q(\RN)}^q -\|r^0_{n}\|_{L^q(\RN)}^q = o(1)$ as $n \toinfty$.
It is noteworthy that such a decomposition of integral functionals is realized as the strict equality (cf.~\eqref{eq;energy estim Wmp}).
This result has been obtained by some precursors on profile decomposition, e.g.,~\cite{T-01,T-K}. 
Furthermore, we shall extend this result to lower order derivatives of $u_n$ and profiles. This extension will be shown with the help of \cref{lemma;G-complconti-lowerorderderiv}.

\subsection{Brief Summary}
Before investigating the decomposition of integral functionals below, we shall briefly review our results, picking up much typical and simple examples in order to describe the essence of the decomposition of integral functionals. 

Assume that $(u_n)$ is a bounded sequence in $\Wmp$ and take profile elements $(\bbl{w}{l}, \dsl{y}{l}{n}, \Lambda)$ (taking a subsequence still denoted by $n$) given by \cref{theorem;profile-decomp. in Wmp}.
According to \cref{lemma;G-complete continuity of Wmp,lemma;G-complconti-lowerorderderiv,lemma;G-complconti-lowerorderderiv2} and the exactness condition~\eqref{eq;Ishiwata condition Wmp}, the residual term $r^L_n$ and its derivatives are vanishing (as $n\toinfty$ and then $L\to\Lambda$) in suitable $L^q(\RN)$ or $W^{k,p}(\RN)$. 

Typically, our main results read: 
\begin{alignat}{2}
&\lim_{n\to\infty}\int_{\RN} |u_n|^q \,\dd x
=\sum_{l=0}^\Lambda \int_{\RN} |\bbl{w}{l}|^q \,\dd x, &\quad 
&q\in ]p,\ppm[, \label{massdeco1}\\ 
&\lim_{n\toinfty}\int_{\RN}|\d^\alpha u_n|^q \,\dd x 
=\sum_{l=0}^\Lambda \int_{\RN} |\d^\alpha \bbl{w}{l}|^q \,\dd x, 
&\quad &q\in]p,p^*_{m-|\alpha|}[, \ |\alpha|<m, \label{massdeco2}\\ 
&\lim_{n\toinfty}\|u_n\|_{W^{k,q}(\RN)}^q
=\sum_{l=0}^\Lambda \|\bbl{w}{l}\|_{W^{k,q}(\RN)}^q, &\quad 
&q\in]p,p^*_{m-k}[, \ 0\le k<m. \label{massdeco3}
\end{alignat}

These formulas are shown as follows.  
As for Euclidean norms, one can show that  
\begin{alignat}{2}
&\l| \l|\sum_{l=1}^L s_l\r|^q-\sum_{l=1}^L|s_l|^q\r|
\le C_L \sum_{l\neq k}|s_l||s_k|^{q-1}, &\quad &s_l\in\R, \label{massdecoineq1}\\ 
&| |a+b|^q-|a|^q|\le \eps |a|^q+C_\eps|b|^q, &\quad &\eps>0, \ a,b\in\R. \label{massdecoineq2}
\end{alignat}
Combining the above, we have: 
\begin{align*}
&\l|\int_{\RN}|u_n|^q\,\dd x 
-\sum_{l=0}^L\int_{\RN}|\bbl{w}{l}|^q\,\dd x \r| \\
&=\l|\int_{\RN}|u_n|^q\,\dd x 
-\sum_{l=0}^L\int_{\RN}|\bbl{w}{l}(\cdot-\dsl{y}{l}{n})|^q\,\dd x \r| \\
&\le \eps\int_{\RN}\l|\sum_{l=0}^L\bbl{w}{l}(\cdot-\dsl{y}{l}{n})\r|^q\,\dd x  
+C_\eps\int_{\RN}|r^L_n|^q\,\dd x \\
&\qquad \qquad \qquad \qquad +C_L\sum_{l\neq k}\int_{\RN}|\bbl{w}{l}(\cdot-\dsl{y}{l}{n})||\bbl{w}{k}(\cdot-\dsl{y}{k}{n})|^{q-1}\,\dd x.
\end{align*}
Since $(u_n)$ and $(r^L_n)$ are bounded in $\Wmp$, 
$\int_{\RN}|\sum_{l=0}^L\bbl{w}{l}(\cdot-\dsl{y}{l}{n})|^q\,\dd x$ is bounded. 
By the mutual orthogonality condition, $\int_{\RN}|\bbl{w}{l}(\cdot-\dsl{y}{l}{n})||\bbl{w}{k}(\cdot-\dsl{y}{k}{n})|^{q-1}\,\dd x$ is converging to zero as $n\toinfty$. 
Also, the residual term $\int_{\RN}|r^L_n|^q\,\dd x$ is vanishing as $n\toinfty$ and then $L\to\Lambda$. 
Therefore, passing to the limits as $n\toinfty$, $L\to\Lambda$ and then, $\eps\to 0$, one can conclude~\eqref{massdeco1}.
Similarly, one can show~\eqref{massdeco2} and~\eqref{massdeco3}.

In the above observation, the inequality~\eqref{massdecoineq1} is employed together with the mutual orthogonality condition, yielding the degeneracy of cross terms. The inequality~\eqref{massdecoineq2} is used with $b=r^L_n$ which is vanishing in suitable spaces. 
Hence, one can obtain other variants of the above iterated Brezis-Lieb lemma, by considering the Lebesgue or Sobolev spaces into which $\Wmp$ is embedded $G[\ZN]$-completely continuously. 
Along the above strategy, in what follows, we shall investigate more general results of decomposition of integral functionals.

\subsection{Main theorems}
Now we shall discuss general results of decomposition of integral functionals of $(\intrn F(\d^\alpha u_n) \,\dd x )$ with continuous functions $F$ of $p$-superlinear and subcritical growth (see also~\eqref{growth condition of F} below).  
We refer the reader to~\cite[Section~3.6]{T-K} as a reference of conditions of continuous functions.

\begin{theorem}[Decomposition of integral functionals~1]
\label{theorem;Lp estim.}
Let $(u_n)$ be a bounded sequence in $\Wmp$ 
and suppose, on a subsequence (still denoted by $n$), that $(u_n)$ has a profile decomposition with profile elements 
$ (\bbl{w}{l}, \dsl{y}{l}{n}, \Lambda ) \in \Wmp \times \ZN\times (\N\cup\{0,\infty\})$ 
$(l \in \NN^{<\Lambda +1}$, $n \in  \Z_{ \ge l})$ 
as in \cref{theorem;profile-decomp. in Wmp}.
Let $\alpha\in(\NN)^N$ be a multi-index satisfying $|\alpha|<m$. 
Assume that $f_\alpha \in C(\R; \R)$ satisfies that 
for any $\eps >0$,  there exist $q = q_{\alpha,\eps} \in ]p,p^*_{m-|\alpha|}[$ and $C = C_{\alpha,\eps} >0$ such that 
\begin{equation}\tag{f}\label{growth condition of f}
|f_\alpha(s)| \le 
\begin{cases}
\eps (|s|^{p-1}+|s|^{(p^*_{m-|\alpha|})-1}) + C|s|^{q-1},  &s \in \R, \quad \mbox{if}\ p^*_{m-|\alpha|} <\infty, \\
\eps |s|^{p-1} + C|s|^{q-1}, &s \in \R,  \quad \mbox{if}\ p^*_{m-|\alpha|} = \infty,   
\end{cases}
\end{equation}
and define a primitive function $F_\alpha$ of $f_\alpha$ by
$F_\alpha(s) = \int_0^s f_\alpha(t)\,\dd t,$ $s \in \R.$
Then the following relation holds true: 
\begin{equation}\label{eq;Lp estim.}
\lim_{n \toinfty} \int_{\RN} F_\alpha(\d^\alpha u_n)  \,\dd x  
= \sum_{l=0}^\Lambda \int_{\RN} F_\alpha(\d^\alpha w^{(l)})  \,\dd x . 
\end{equation}
\end{theorem}

\begin{remark}
\rm 
Regarding a growth condition of the primitive function $F_\alpha$, it follows from~\eqref{growth condition of f} that  $F_\alpha$ satisfies that for any $\eps >0$,  there exist $q=q_{\alpha,\eps} \in ]p,p^*_{m-|\alpha|}[$  and  $C=C_{\alpha,\eps}>0$ such that 
\begin{equation}\tag{F}\label{growth condition of F}
|F_\alpha(s)| \le 
\begin{cases}
\eps (|s|^{p}+|s|^{p^*_{m-|\alpha|}}) + C|s|^{q},  &s \in \R,\quad \mbox{if}\ p^*_{m-|\alpha|} <\infty, \\
\eps |s|^{p} + C|s|^{q}, &s \in \R,  \quad \mbox{if}\ p^*_{m-|\alpha|}= \infty.   
\end{cases}
\end{equation}
Here one can take $q_{\alpha,\eps}$ above as the same exponent as in~\eqref{growth condition of f}.
\end{remark}

Secondly, we shall generalize conditions of continuous functions, and consequently, we obtain the same results.

\begin{theorem}[Decomposition of integral functionals~2]
\label{theorem;Lp estim.2}
Let $(u_n)$ be a bounded sequence  in $\Wmp$ 
and suppose that, on a subsequence (still denoted by $n$), $(u_n)$ has a profile decomposition with profile elements 
$(\bbl{w}{l}, \dsl{y}{l}{n}, \Lambda )\in \Wmp \times \ZN\times (\N\cup\{0,\infty\})$   
$( l \in \NN^{<\Lambda +1}$, $n \in \Z_{ \ge l})$
as in \cref{theorem;profile-decomp. in Wmp}. 
Let $\alpha\in(\NN)^N$ be a multi-index satisfying $|\alpha|<m$.
\begin{enumerate}
    \item Suppose that $f_\alpha(x,s) \in C(\RN \! \times \R;\R)$ satisfies the following growth condition in $s$ and invariance condition in $x$: 
    \begin{equation}\tag{f'}\label{growth cond. of f 2.1}
    \l\{
    \begin{aligned}
    &\bullet \forall\, \eps >0,\  \exists\, q=q_{\alpha,\eps} \in ]p,p^*_{m-|\alpha|}[, \ \exists\,  C=C_{\alpha,\eps}>0 \ \mbox{s.t.} \ \\
    & \sup_{x\in\RN}|f_\alpha(x,s)|\le
\begin{cases}
\eps (|s|^{p-1}+|s|^{(p^*_{m-|\alpha|})-1}) + C|s|^{q-1},  & \mbox{if}\  p^*_{m-|\alpha|}<\infty, \\
\eps |s|^{p-1} + C|s|^{q-1}, & \mbox{if}\ p^*_{m-|\alpha|}= \infty,   
\end{cases} \\
    &\bullet f_\alpha(x+y,s)=f_\alpha(x,s), 
    \s y \in \ZN, \  x \in \RN, \  s \in \R.
    \end{aligned} 
    \r.
    \end{equation}
    Then for a primitive function 
    \begin{equation*}
    F_\alpha(x,s) \coloneqq \int_0^s f_\alpha(x,t) \,\dd t, \s x \in \RN, \ s \in \R,
    \end{equation*}
    the following relation holds true: 
    \begin{equation*}
    \lim_{n \toinfty} \int_{\RN} F_\alpha(x,\d^\alpha u_n)  \,\dd x  
    = \sum_{l=0}^\Lambda \int_{\RN} F_\alpha(x,\d^\alpha w^{(l)})  \,\dd x .
    \end{equation*}
\item Suppose that  $f_\alpha(x,s) \in C(\RN \! \times \R;\R)$ satisfies the following growth condition in $s$ and limit state condition in $x$: 
    \begin{equation}\tag{f''}\label{growth cond. of f 2.2}
    \l\{
    \begin{aligned}
    &\bullet \forall\, \eps >0, \ \exists\, q=q_{\alpha,\eps} \in ]p,p^*_{m-|\alpha|}[, \ \exists\, C=C_{\alpha,\eps}>0  \ \mbox{s.t.}  \\
    & \sup_{x\in\RN}|f_\alpha(x,s)|\le
\begin{cases} 
\eps (|s|^{p-1}+|s|^{(p^*_{m-|\alpha|})-1}) + C|s|^{q-1},  &\!\! \mbox{if}\ p^*_{m-|\alpha|} <\infty, \\
\eps |s|^{p-1} + C|s|^{q-1}, &\!\! \mbox{if}\ p^*_{m-|\alpha|} = \infty,   
\end{cases} \\
&\bullet f_{\alpha,\infty}(s) \coloneqq \lim_{|x|\toinfty}f_\alpha(x,s) \s \mbox{exists for all} \  s \in \R.
\end{aligned} 
\r.
\end{equation}
Then for primitive functions 
\begin{equation}\notag
\l\{
\begin{alignedat}{2}
F_\alpha(x,s) &\coloneqq \int_0^s f_\alpha(x,t)\,\dd t, &\quad &x \in \RN, \s s \in \R,\\
F_{\alpha,\infty}(s) &\coloneqq \int_0^s f_{\alpha,\infty}(t)\,\dd  t, &\quad  &s \in \R,\\
\end{alignedat}
\r.
\end{equation}
the following relation holds true: 
\begin{equation}\notag 
\lim_{n \toinfty} \int_{\RN} F_\alpha(x,\d^\alpha u_n)  \,\dd x  
= \int_{\RN}F_\alpha(x,\d^\alpha \bbl{w}{0} ) \,\dd x  + \sum_{l=1}^\Lambda  \int_{\RN} F_{\alpha,\infty}(\d^\alpha \bbl{w}{l})  \,\dd x . 
\end{equation}
\end{enumerate}
\end{theorem}

One can prove many other variants of the above results in the same way, and we shall exhibit the following typical variants.

\begin{proposition}\label{prop;massdecomp1}
Along the profile decomposition in $W^{m,p}(\RN)$ $(m>1)$, there hold 
\begin{alignat*}{2}
&\lim_{n\toinfty}\int_{\RN} |\nabla u_n|^q \,\dd x = \sum_{l=0}^\Lambda \int_{\RN} |\nabla \bbl{w}{l}|^q \,\dd x, &\quad &q\in ]p,p^*_{m-1}[, \\ 
&\lim_{n\toinfty}\|u_n\|_{W^{k,q}(\RN)}^q=\sum_{l=0}^\Lambda \|\bbl{w}{l}\|_{W^{k,q}(\RN)}^q, &\quad &q\in]p,p^*_{m-k}[, \  k\in\NN, \ k<m.
\end{alignat*}
\end{proposition}

\subsection{Proof of \cref{theorem;Lp estim.}}
In this and the following subsections, 
we shall prove \cref{theorem;Lp estim.,theorem;Lp estim.2} without relying on pointwise convergence. 
We shall prove the results of decomposition of integral functionals above in the special case $|\alpha|=0$ and we abbreviate $f_\alpha$ to $f$ and so on; the other cases will be readily shown in the same way.
Note that proofs provided below are concerned with only the case $\ppm <\infty$. The other case $\ppm = \infty$ can also be proved in the same manner. 

\paragraph{Basic lemmas}
First of all, we shall prove the following important inequalities which correspond to~\eqref{massdecoineq1} and~\eqref{massdecoineq2}. 

\begin{lemma}\label{lemma;power estimates of F 2}
Assume the same conditions as in \cref{theorem;Lp estim.}. Then for any $\eps >0$, there exist $q_1=q_{1,\eps} \in ]p,\ppm[$ and $C=C_\eps>0$ such that for all $a,b \in \R$,  
\begin{align}
\label{ineq;power estimates of F 2}
&\left|F (a+b)-F(a) \right| 
\\
&\notag 
\le \eps \left( |a|^{p-1}|b|+|a|^{\ppm-1}|b|+|b|^p+|b|^{\ppm} \right) 
+C (|a|^{q_1-1}|b|+|b|^{q_1}). 
\end{align}
\end{lemma}
    
\begin{proof}
Fix $\delta>0$ arbitrarily. 
For any $a,b\in\R$, the mean value theorem implies 
\begin{equation*}
    \left|F (a+b)-F(a) \right|
    =|f(a+\theta b)||b|
\end{equation*}
for some $\theta = \theta_{a,b} \in [0,1]$.
It follows from~\eqref{growth condition of f} that 
\begin{align}
\label{202103170010}
    &\left|F (a+b)-F(a) \right| \\
    &\notag \le 
    \qty[ \delta ( |a+\theta b|^{p-1}+|a+\theta b|^{\ppm-1}  ) + C_\delta |a+\theta b|^{q_\delta-1}   ]|b| \\
    &\notag \le
    K \delta  \qty(|a|^{p-1}+|a|^{\ppm-1}+|b|^{p-1}+|b|^{\ppm-1})|b|+ C_{\delta} (|a|^{q_\delta-1}|b|+|b|^{q_\delta}), 
\end{align}
where $K>0$ is independent of $\delta$ and $q_\delta$. 
Then for any $\eps>0$, set $\delta = \eps/K$ and we then observe from~\eqref{202103170010} that 
\begin{align*}
    &\left|F (a+b)-F(a) \right| \\
    &\le \eps \qty(|a|^{p-1}+|a|^{\ppm-1}
    +|b|^{p-1}+|b|^{\ppm-1})|b|
    +C_{\eps} (|a|^{q_{\eps/K}-1}|b|+|b|^{q_{\eps/K}}).
\end{align*}
Thus letting $q_1 \coloneqq q_{\eps/K}$, one concludes~\eqref{ineq;power estimates of F 2}.
\end{proof}

\begin{lemma}\label{lemma;power estimates of F}
Assume the same conditions as in \cref{theorem;Lp estim.}. Let $L \in \N$, fix $\eps >0$ arbitrarily and let $q=q_\eps \in ]p,\ppm[$ be an exponent corresponding to~\eqref{growth condition of f}. Then there exists $C_{\eps,L} >0$ such that 
for all $s_l \in \R, \  1 \le l \le L$,
\begin{equation}\label{ineq;power estimates of F}
\left|F\left(\sum_{l=1}^L s_l \right)-\sum_{l=1}^L F(s_l)\right| 
\le \eps C_L   \left(\sum_{l=1}^L(|s_l|^p+|s_l|^{\ppm}) \right) 
+ C_{\eps,L} \sum_{1\le l\neq k\le L} |s_l|^{q_\eps-1}|s_k| 
\end{equation}
for some constant $C_L>0$. 
\end{lemma}

\begin{proof}
We prove \eqref{ineq;power estimates of F} by  induction on $L$.

\paragraph{(I) Base step: $L=2$}
We shall prove that for any $a,b \in \R$,  
\begin{align}
\label{ineq;power estim. of F of (I)}
&\left|F(a+b)-F(a)-F(b)\right| \\
&\notag \le 
C \eps (|a|^p+|a|^{\ppm}+|b|^p+|b|^{\ppm})
+C_\eps (|a|^{q_\eps-1}|b|+|a||b|^{q_\eps -1}) 
\end{align}
for some positive constant $C$ independent of $\eps$ and $q_\eps$. 
Since $F(0)=0$,~\eqref{ineq;power estim. of F of (I)} is true if $a=0$ or $b=0$. 
So we assume $a,b \neq 0$ and thanks to the symmetry, assume that  
$t=|b|/|a| \le 1$. Then by the mean value theorem,
we have
\begin{equation}\notag
F(a+b)-F(a) = f(a+\theta b)b \quad \mbox{for some} \s \theta = \theta_{a,b} \in [0,1].
\end{equation}
It follows from~\eqref{growth condition of f} and~\eqref{growth condition of F} that  
\begin{align*}
&|F(a+b)-F(a)-F(b)| \\ 
&\le |f(a+\theta b)||b|+|F(b)|\\
&\le \qty[\eps (|a+\theta b|^{p-1} +|a+\theta b|^{\ppm-1}) +C_\eps  |a+\theta b|^{q_\eps -1}]|b|
\\
&\notag \qquad 
+\eps(|b|^p+|b|^{\ppm})+C_\eps |b|^{q_\eps} \\
&\le C \eps  (|a|^p+|a|^{\ppm})
+C_\eps |a|^{q_\eps -1}|b|,    
\end{align*}
where we used the assumption $|b| \le |a|$ in the last inequality.
Thus \eqref{ineq;power estim. of F of (I)} is verified.

\paragraph{(II) Inductive step}
Assume that 
\eqref{ineq;power estimates of F} is valid for some $L \in \N$.
Setting 
$a=\sum_{l=1}^L s_l$ and using~\eqref{ineq;power estim. of F of (I)} and the induction hypothesis, we observe that 
\begin{align*}
&\left|F\left(\sum_{l=1}^{L+1} s_l \right)-\sum_{l=1}^{L+1} F(s_l)\right| 
=\left|F\left(a+ s_{L+1} \right)-\sum_{l=1}^{L+1} F(s_l)\right|\\
&\le \left|F\left(a+ s_{L+1} \right) -F(a)-F(s_{L+1}) \right|
+ \left|F(a)-\sum_{l=1}^{L} F(s_l)\right|\\
&\le \eps C   (|a|^p+|a|^{\ppm}+|s_{L+1}|^p+|s_{L+1}|^{\ppm})
+C_{\eps} (|a|^{q_\eps-1}|s_{L+1}|+|a||s_{L+1}|^{q_\eps -1})\\
 &\qquad \qquad \qquad + \eps C_L    \left(\sum_{l=1}^L(|s_l|^p+|s_l|^{\ppm}) \right) 
+C_{\eps,L} \sum_{1\le l\neq k\le L} |s_l|^{q_\eps -1}|s_k|\\
&\le  \eps  C_{L+1}  \left(\sum_{l=1}^{L+1}(|s_l|^p+|s_l|^{\ppm}) \right) 
+C_{\eps, L+1} \sum_{1\le l\neq k \le L+1} |s_l|^{q_\eps -1}|s_k|.    
\end{align*}
Thus \eqref{ineq;power estimates of F} holds when $L$ is replaced by $L+1$. 
This completes the proof. 
\end{proof}

We need a modification to the preceding lemma. 

\begin{lemma} \label{lemma;lemma;power estimates of F modified}
Assume the same conditions as in \cref{theorem;Lp estim.}. Let $L\in\N$. Then for any $\eps>0$, there exist $q_L = q_{\eps,L} \in ]p,\ppm[$ and $C=C_{\eps,L}>0$ such that for all $s_l \in \R$, $1\le l \le L$, 
\begin{equation}\notag 
\left|F\left(\sum_{l=1}^L s_l \right)-\sum_{l=1}^L F(s_l)\right| 
\le \eps  \left(\sum_{l=1}^L(|s_l|^p+|s_l|^{\ppm}) \right) 
+ C \sum_{1\le l\neq k\le L} |s_l|^{q_L-1}|s_k|.
\end{equation}
\end{lemma}

\begin{proof}
For any $\eps>0$, set $\delta = \eps/C_L$, where $C_L$ is as in \cref{lemma;power estimates of F}, and we observe from~\eqref{ineq;power estimates of F} that 
\begin{align}\notag 
\left|F\left(\sum_{l=1}^L s_l \right)-\sum_{l=1}^L F(s_l)\right| 
&\le \delta  C_L \left(\sum_{l=1}^L(|s_l|^p+|s_l|^{\ppm}) \right) + C_{\delta,L} \sum_{1\le l\neq k\le L} |s_l|^{q_\delta-1}|s_k| \\
&=\eps \left(\sum_{l=1}^L(|s_l|^p+|s_l|^{\ppm}) \right) + C_{\eps,L} \sum_{1\le l\neq k\le L} |s_l|^{q_{L}-1}|s_k|,
\end{align}
where $q_L = q_{\eps,L} \coloneqq q_{\eps/{C_L}}$. 
Thus one can conclude~\eqref{lemma;lemma;power estimates of F modified}.
\end{proof}

Combining \cref{lemma;power estimates of F 2,lemma;lemma;power estimates of F modified} with the Young inequality, one can obtain the following 

\begin{lemma}
Let $F$ be as in \cref{theorem;Lp estim.} and let $L \in \N$. Then for any $\eps>0$, there exist $q_1=q_{1,\eps} \in ]p, \ppm[$, $q_L=q_{L,  \eps } \in ]p,  \ppm[$ and $C_{\eps}, C_{\eps,L} >0$ such that 
for all $s_l, r \in \R$ $(l=1,\ldots, L)$, 
\begin{align}
\label{ineq;100}
    &\l|F \l(\sum_{l=1}^L s_l +r \r) - \sum_{l=1}^L F(s_l)  \r| \\ 
    &\notag \le 
    \eps C_p \l(
        \l| \sum_{l=1}^L s_l \r|^p + \l| \sum_{l=1}^L s_l \r|^{\ppm} + \sum_{l=1}^L |s_l|^p + \sum_{l=1}^L |s_l|^{\ppm} + |r|^p  + |r|^{\ppm} 
    \r) \\ 
    &\notag \qquad +C_\eps \left( 
        |r|^{q_1} +  \l| \sum_{l=1}^L s_l \r|^{q_1-1} |r| \right)
        + C_{\eps,L} \sum_{1 \le l \neq l' \le L} |s_l||s_{l'}|^{q_L-1},
\end{align}
where $C_p >0$ depends only on $p$ and $C_\eps$ depends only on $\eps$. 
\end{lemma}

\paragraph{Main body}
Using the above inequalities we can prove \eqref{eq;Lp estim.}.

\begin{proof}[Proof of \cref{theorem;Lp estim.}]
Set 
\begin{equation}\notag
E\coloneqq \sup_{n \in \NN}\|u_n\|_{\Wmp}, \quad 
S^L_n \coloneqq \sum_{l=0}^L \bbl{w}{l}(\cdot-\dsl{y}{l}{n}), \quad  L \in \NN^{<\Lambda +1}, \ 
n \in \Z_{\ge L}, 
\end{equation}
and then one has the relation $u_n = S^L_n + r^L_n$. 
Fix $\eps>0$ arbitrarily and let $q_1, q_L \in  ]p,\ppm[$ be exponents given by \cref{lemma;power estimates of F 2,lemma;lemma;power estimates of F modified}.
Using~\eqref{ineq;100} and the H\"older inequality, one sees that 
\begin{align}
&\label{10170060} 
\int_{\RN}\l| F(u_n) - \sum_{l=0}^L F(\bbl{w}{l}(\cdot - \dsl{y}{l}{n}) ) \r|  \,\dd x   \\
&\le 
\eps C_p 
\l(
    \|S^L_n \|_{L^p(\RN)}^p 
    + \| S^L_n \|_{L^{\ppm}(\RN)}^{\ppm} 
    + \sum_{l=0}^L \|\bbl{w}{l} \|_{L^p(\RN)}^p \r. \notag \\
&\hspace{2cm}   \l. + \sum_{l=0}^L \|\bbl{w}{l} \|_{L^{\ppm}(\RN)}^{\ppm} 
    +\| r^L_n \|_{L^p(\RN)}^p 
    +\| r^L_n \|_{L^{\ppm}(\RN)}^{\ppm} 
\r) \notag  \\
&\qquad + C_\eps 
\l(
    \| r^L_n \|_{L^{q_1}(\RN)}^{q_1} 
    + \l\| S^L_{n} \r\|_{L^{q_1}(\RN)}^{q_1-1} \| r^L_n \|_{L^{q_1}(\RN)} \r) \notag \\ 
&\hspace{2cm} + C_{\eps,L}  \sum_{0 \le l \neq l' \le L} 
    \int_{\RN}  |\bbl{w}{l}(x - \dsl{y}{l}{n})  ||\bbl{w}{l'}(x - \dsl{y}{l'}{n})  |^{q_L-1}  \,\dd x .
 \notag 
\end{align}

According to the estimates~\eqref{eq;residue bounded Wmp},~\eqref{202105120020} 
and the Sobolev inequality,
one gets for any $r\in[p,\ppm]$,  
\begin{equation}\label{10170070}
\l\{
\begin{aligned}
\varlimsup_{L \to \Lambda} \varlimsup_{n \toinfty} 
\|r^L_n\|_{L^{r}(\RN)}
&\le C_r 
\varlimsup_{L \to \Lambda} \varlimsup_{n \toinfty} 
\|r^L_n\|_{\Wmp}
\le C_{r} E, \\
\varlimsup_{L \to \Lambda} \varlimsup_{n \toinfty} 
\|S^L_n\|_{L^{r}(\RN)}
&\le C_r 
\varlimsup_{L \to \Lambda} \varlimsup_{n \toinfty} 
\|S^L_n\|_{\Wmp} 
\le C_r E.
\end{aligned}
\r.
\end{equation}
From~\eqref{eq;energy estim Wmp} one has 
\begin{equation}\label{202103170030}
\sum_{l=0}^\Lambda  \|\bbl{w}{l}\|_{L^p(\RN)}^p
\le E^p.
\end{equation}
From the Sobolev inequality and the embedding $\ell^p(\N) \hookrightarrow  \ell^{\ppm}(\N)$, one also has 
\begin{align}
\label{202103170040}
\sum_{l=0}^\Lambda \|\bbl{w}{l}\|_{L^{\ppm}(\RN)}^{\ppm}
&\le C_{\ppm} \sum_{l=0}^\Lambda  \|\bbl{w}{l}\|_{W^{m,p}(\RN)}^{\ppm} \\ 
&\notag \le C_{\ppm} \l( \sum_{l=0}^\Lambda  \|\bbl{w}{l}\|_{W^{m,p}(\RN)}^{p} \r)^{\ppm/p} \\
&\notag \le C_{\ppm} E^{\ppm}.
\end{align}
Also, one has the condition~\eqref{eq;residue vanish Wmp}, i.e., 
\begin{equation}\label{202105080010}
\varlimsup_{L\to \Lambda} \varlimsup_{n\toinfty} \| r^L_n \|_{L^q(\RN)}=0\quad \mbox{for any}\ q \in ]p, \ppm[.
\end{equation}

Moreover, one can check that for any $0 \le l \neq l' \in \NN^{<\Lambda+1}$ and for any $r \in [p,\ppm]$,  
\begin{equation}\label{202105080020}
\int_{\RN} |\bbl{w}{l}(x-\dsl{y}{l}{n})||\bbl{w}{l'}(x-\dsl{y}{l'}{n})|^{r-1}  \,\dd x  \to 0
\end{equation}
as $n\toinfty$. 
Combining~\eqref{10170060}--\eqref{202105080020}, and passing to the limits as $n\toinfty$, $L\to \Lambda$, and then $\eps \to 0$, one obtains 
\begin{equation}\label{202105080030}
\varlimsup_{L \to \Lambda} \varlimsup_{n\toinfty}
\int_{\RN}\l| F(u_n) - \sum_{l=0}^L F(\bbl{w}{l}(\cdot - \dsl{y}{l}{n}) ) \r|  \,\dd x  
=0.
\end{equation}
From this follows 
\begin{equation*}
\varlimsup_{L \to \Lambda} \varlimsup_{n\toinfty}
\l| 
\int_{\RN}  F(u_n)   \,\dd x  - \sum_{l=0}^L \int_{\RN} F(\bbl{w}{l})    \,\dd x  
\r|   
=0.
\end{equation*}
Eventually, the limits 
\begin{equation*}
\lim_{n\toinfty} \int_{\RN}  F(u_n)   \,\dd x , \quad 
\lim_{L \to \Lambda} \sum_{l=0}^L \int_{\RN} F(\bbl{w}{l})    \,\dd x 
\end{equation*}
exist and it follows that 
\begin{equation*}
\lim_{n\toinfty} \int_{\RN}  F(u_n)   \,\dd x  
=
\sum_{l=0}^\Lambda \int_{\RN} F(\bbl{w}{l})    \,\dd x .
\end{equation*}
This completes the proof of \cref{theorem;Lp estim.}.
\end{proof}

\subsection{Proof of \cref{theorem;Lp estim.2}}
We move on to the proof of \cref{theorem;Lp estim.2} 
which is proved based on an argument quite similar to  the preceding one.
As is similar to the preceding subsection, it is sufficient to show the  results of decomposition of integral functionals in the special case $|\alpha|=0$ and $p^*_{m}<\infty$, and we shall abbreviate $f_\alpha$ to $f$ and so on.
In what follows, we assume that $u_n, \bbl{w}{l}$ and $\dsl{y}{l}{n}$ are as in 
\cref{theorem;Lp estim.2}.

\paragraph{Assertion~(i)}
Let $f$ satisfy \eqref{growth cond. of f 2.1}, 
and let $F$ be defined by 
$$
F(x,s)= \int_0^s f(x,t) \,\dd t, \quad x \in \RN, \ s \in \R.
$$
One can show the following counterpart of 
\cref{lemma;power estimates of F 2}
in the same manner:
$\forall\, \eps >0, \  \exists\, q_1=q_{1,\eps} \in]p,\ppm[,  \ \exists\, C= C_\eps>0$ s.t.
$\forall\, a, b \in \R$, 
\begin{align}
\label{eq;prTheorem Lp estim.2 (i) 3}
&\sup_{x\in\RN} \left|F\left(x, a+b \right)
- F(x, a )\right|  \\
&\notag \le \eps \left( |a|^{p-1}|b|+|a|^{\ppm-1}|b|+|b|^p+|b|^{\ppm}  \right) 
+C ( |a|^{q_1-1}|b| + |b|^{q_1}).
\end{align}
One can also show the following counterpart of 
\cref{lemma;lemma;power estimates of F modified}
in the same manner:
$\forall\, \eps >0, \ \forall\, L\in\N, \  \exists\, q_L=q_{L,\eps} \in ]p,\ppm[, \  \exists\, C=C_{\eps,L}>0$ s.t. $\forall\, s_l\in\R$, $1\le l \le L$, 
\begin{align}
\label{eq;prTheorem Lp estim.2 (i) 1}
&\sup_{x\in\RN} \left|F\left(x,\sum_{l=1}^L s_l \right)
-\sum_{l=1}^L F(x,s_l)\right|  \\
&\notag \le \eps \left(\sum_{l=1}^L(|s_l|^p+|s_l|^{\ppm}) \right) 
+C \sum_{1\le l\neq k\le L} |s_l|^{q_L-1}|s_k|. 
\end{align}

Thanks to these two inequalities, one can verify the counterpart of~\eqref{202105080030} in the same manner, 
where one replaces $F(s)$ by $F(x,s), \ x \in \RN$,  i.e., 
\begin{equation}\notag 
\lim_{L\to \Lambda}\varlimsup_{n\toinfty} \int_{\RN}
\l| 
F(x,u_n)-\sum_{l=0}^L F(x,\bbl{w}{l}(x-\dsl{y}{l}{n})) 
\r|  \,\dd x 
=0.
\end{equation}
Owing to the $\ZN$-invariance of $f$ in $x$ and the change of variables, it follows that 
\begin{align}\notag
&\varlimsup_{L\to \Lambda}\varlimsup_{n\toinfty}
\l| \int_{\RN} F(x,u_n)  \,\dd x  - \sum_{l=0}^L \int_{\RN} F(x,w^{(l)})  \,\dd x  \r|\\
&=
\varlimsup_{L\to \Lambda}\varlimsup_{n\toinfty}
\l| \int_{\RN} F(x,u_n)  \,\dd x  - \sum_{l=0}^L \int_{\RN} F(x, \bbl{w}{l}(x-\dsl{y}{l}{n})  )  \,\dd x  \r|\\
&\le 
\varlimsup_{L\to \Lambda}\varlimsup_{n\toinfty}
\int_{\RN} \l| F(x,u_n) -\sum_{l=0}^L F( x,\bbl{w}{l}(x-\dsl{y}{l}{n})  )  \r|  \,\dd x  \\
&=0,
\end{align}
which yields 
\begin{equation}\notag
\lim_{n\toinfty} \int_{\RN} F(x,u_n)  \,\dd x  
= 
\sum_{l=0}^\Lambda \int_{\RN} F(x,w^{(l)})  \,\dd x ,  
\end{equation}
hence the conclusion.

\paragraph{Assertion~(ii)}
Let $f$ and $f_{\infty}$ satisfy \eqref{growth cond. of f 2.2}, and let $F $ and $F_{\infty}$ be defined by 
\begin{equation}\notag
\begin{alignedat}{2}
F(x,s) &= \int_0^s f(x,t) \,\dd t, &\quad &x \in \RN, \s s \in \R, \\
F_\infty (s) &= \int_0^s f_\infty (t) \,\dd t, &\quad &s \in \R.
\end{alignedat}
\end{equation}
As in the previous proof,
one can show 
\eqref{eq;prTheorem Lp estim.2 (i) 3} and \eqref{eq;prTheorem Lp estim.2 (i) 1}.
So one can verify the counterpart of~\eqref{202105080030}, 
where one replaces $F(s)$ by $F(x,s), \ x \in \RN$,  i.e., 
\begin{equation}\label{eq;prTheorem Lp estim.2 (i) 7}
\lim_{L\to \Lambda} \varlimsup_{n\toinfty} \int_{\RN}
\l| 
F(x,u_n)-\sum_{l=0}^L F(x,\bbl{w}{l}(x-\dsl{y}{l}{n})) 
\r| \,\dd x 
=0. 
\end{equation}
By changing variables, we get 
\begin{align}\notag
\int_{\RN} F(x,\bbl{w}{l}(x-\dsl{y}{l}{n}))  \,\dd x  
&=
\int_{\RN} F(x+\dsl{y}{l}{n},\bbl{w}{l}(x))  \,\dd x . \\
\end{align}

Since $|\dsl{y}{l}{n}| \toinfty$ for $l \ge 1$,  
it follows that 
\begin{equation}\notag 
F(x+\dsl{y}{l}{n},\bbl{w}{l}(x))
\to F_\infty (\bbl{w}{l}(x)) 
\quad \mbox{a.e.} \ x \in \RN \  (n\toinfty, \ l \ge 1).  
\end{equation}
Here we used the fact that 
\begin{equation}\notag 
F(x,s) \to F_\infty (s)
\quad \mbox{as \s} |x| \toinfty \ \mbox{for any} \ s \in \R,  
\end{equation}
which is proved via \eqref{growth cond. of f 2.2}
and the dominated convergence theorem.
Also, due to \eqref{growth cond. of f 2.2}, 
the integrand 
$\dsp |F(x+\dsl{y}{l}{n},\bbl{w}{l}(x))|$ is 
dominated by an integrable function independent of $n$.
Hence, by the dominated convergence theorem, we get
for $l \ge 1$,   
\begin{align}
\label{eq;prTheorem Lp estim.2 (i) 9}
\int_{\RN} F(x,\bbl{w}{l}(x-\dsl{y}{l}{n}))  \,\dd x  
&=
\int_{\RN} F(x+\dsl{y}{l}{n},\bbl{w}{l}(x))  \,\dd x  \\
&\notag =
\int_{\RN} F_\infty(\bbl{w}{l}(x))  \,\dd x  +o(1)
\mbox{\qquad as \s} n\toinfty.
\end{align}
Thus from~\eqref{eq;prTheorem Lp estim.2 (i) 7} and~\eqref{eq;prTheorem Lp estim.2 (i) 9} follow
\begin{align}\notag 
&\varlimsup_{L\to \Lambda}\varlimsup_{n\toinfty}
\l|  \int_{\RN} F(x,u_n)  \,\dd x  -  
\int_{\RN} F(x,\bbl{w}{0})  \,\dd x  -
\sum_{l=1}^L \int_{\RN} F_\infty (\bbl{w}{l})  \,\dd x  \r| \\
&=
\varlimsup_{L\to \Lambda}\varlimsup_{n\toinfty}
\l| 
 \int_{\RN} F(x,u_n)  \,\dd x  -  
\sum_{l=0}^L \int_{\RN} F(x,\bbl{w}{l}(x-\dsl{y}{l}{n}))  \,\dd x  
+o(1) \r| \\
&\le
\varlimsup_{L\to \Lambda}\varlimsup_{n\toinfty} 
\int_{\RN} \l| F(x,u_n) -  
\sum_{l=0}^L  F (x,\bbl{w}{l}(x-\dsl{y}{l}{n})) \r|  \,\dd x  \\
&=0, 
\end{align}
which yields 
\begin{equation}\notag  
\lim_{n\toinfty} 
\int_{\RN} F(x,u_n)  \,\dd x  
=
\int_{\RN} F(x,\bbl{w}{0})  \,\dd x  +
\sum_{l=1}^\Lambda \int_{\RN} F_\infty (\bbl{w}{l})  \,\dd x .
\end{equation}
Thus the proof is complete.
\qed

\section{Summary: a story of profile decomposition}
As is observed in the present paper, the results of profile decomposition in Sobolev spaces and the results of decomposition of integral functionals are established upon several frameworks and steps. Here we shall make a brief summary, namely, a \emph{recipe} of a theory of profile decomposition and decomposition of integral functionals.
Suppose that $(X,G)$ is a dislocation Sobolev space and that $X$ is $G$-completely continuously embedded into some function space $Y$.
The profile decomposition part consists of mainly three steps: 
(P1) finding profiles; 
(P2) energy decompositions for Sobolev norms and the exactness condition of profile decomposition; 
(P3) $G$-completely continuous embeddings and vanishing of the residual term.
More precisely, Step~(P1) is an application of \cref{theorem;ProDeco in X}. Step~(P2) is performed in a very specific way depending on the norm of $X$. This step will be more complicated, e.g., when $X$ is a homogeneous Sobolev space (see also~\cite{Okumura3}), and as a corollary of the energy estimates, the exactness condition will be obtained. In Step~(P3), vanishing of the residual term is gained by the use of $G$-completely continuous embeddings with the help of the exactness condition. 
On the other hand, the results of decomposition of integral functionals (the iterated Brezis-Lieb lemma) are obtained in the following spirit:  
they are performed in function spaces (e.g., Lebesgue or Sobolev spaces) into which $X$ is embedded $G$-completely continuously.
As are seen in \cref{theorem;Lp estim.,theorem;Lp estim.2,prop;massdecomp1}, they are also obtained for lower order derivatives. 

Hence, roughly speaking, a theory of profile decomposition and decomposition of integral functionals is founded on ``finding a triplet $(X,G,Y)$ where $(X,G)$ is a dislocation space and $X$ is embedded into $Y$ $G$-completely continuously''.

\section*{Acknowledgment}
The author would like to thank his supervisor Goro Akagi for his great support and advice during the preparation of the paper. 
He also wishes to express his gratitude to Michinori Ishiwata and Norihisa Ikoma for sharing their personal notes on the profile decomposition in Hilbert spaces, and also for their valuable comments on reorganizing the present paper. 
He also wishes to thank editors and reviewers who gave him valuable comments to improve this paper. 


\end{document}